\documentclass[a4paper,11pt]{article}

\usepackage{amsfonts,qtree,stmaryrd,textcomp}
\usepackage{pifont,amssymb,amsmath,amsthm,tabularx,graphicx}
\usepackage{tree-dvips,pictexwd,dcpic}
\usepackage{bbm}

\usepackage{stmaryrd}

\usepackage[T1]{fontenc}
\usepackage[latin1]{inputenc}
\usepackage[text={17cm,26cm},centering]{geometry}

\usepackage{etex}
\usepackage{fancyhdr}
\usepackage{graphicx}
\usepackage{enumitem}
\usepackage{amsmath}
\usepackage[bbgreekl]{mathbbol}% order sensitive; it is before amssymb
\usepackage{amssymb}
\usepackage{amsthm}
\usepackage[all]{xy}

\usepackage{epigraph}
\RequirePackage{bussproofs}

\usepackage{framed}
\usepackage{mathtools}
\usepackage{url}
\usepackage{proof}
\usepackage{xspace}
\usepackage{listings}
\usepackage{wrapfig}
\usepackage{tikz}
\usepackage{tikz-cd}

\usepackage{relsize}

\usetikzlibrary{positioning}
\usetikzlibrary{shapes}

\usetikzlibrary{arrows}
\usetikzlibrary{calc}
\usepackage{tikz-3dplot}
\usetikzlibrary{decorations.pathmorphing}

\usepackage{amsfonts,amssymb,amsmath,qtree,stmaryrd,textcomp, amsthm}

\newcommand{\B}{\boldsymbol}
\newcommand{\C}[1]{\mathcal{#1}}
\newcommand{\D}[1]{\mathbb{#1}}
\usepackage{ mathrsfs }

\newtheorem{theorem}{Theorem}[section]
\newtheorem{proposition}[theorem]{Proposition}
\newtheorem{lemma}[theorem]{Lemma}
\newtheorem{corollary}[theorem]{Corollary}

\newtheorem{example}[theorem]{Example}

\newtheorem{definition}[theorem]{Definition}

\newcommand{\Real}{{\mathbb R}}

\newcommand{\Least}{{\mathcal F}}

\newcommand{\BR}{{\Bic(\Real)}}

\newcommand{\id}{\mathrm{id}}

\newcommand{\Const}{\mathrm{Const}}

\newcommand{\Bic}{\mathrm{Bic}}

\newcommand{\BISH}{\mathrm{BISH}}

\newcommand{\CST}{\mathrm{CST}}

\newcommand{\Ind}{\mathrm{Ind}}

\newcommand{\Mor}{\mathrm{Mor}}

\newcommand{\BS}{\mathrm{BS}}

\newcommand{\TOT}{\Leftrightarrow}

\newcommand{\To}{\Rightarrow}

\newcommand{\CZF}{\mathrm{CZF}} 

\newcommand{\REA}{\mathrm{REA}}

\newcommand{\VoF}{\bigvee F_{0}}
\newcommand{\VoG}{\bigvee G_{0}}
\newcommand{\V}{\bigvee}

\newcommand{\sub}{\textnormal{\texttt{sub}}}

\newcommand{\com}{\textnormal{\texttt{abel}}}
\newcommand{\normal}{\textnormal{\texttt{normal}}}
\newcommand{\Normal}{\textnormal{\texttt{Normal}}}
\newcommand{\Center}{\textnormal{\texttt{Center}}}
\newcommand{\Hyp}{\textnormal{\texttt{Hyp}}}
\newcommand{\Goal}{\textnormal{\texttt{Goal}}}

\newcommand{\BTopGrp}{\textnormal{\textbf{BTopGrp}}}

\newcommand{\Ker}{\textnormal{\texttt{Ker}}}

\newcommand{\MOR}{\textnormal{\textbf{Mor}}}

\newcommand{\NS}{\mathrm{NS}}

\begin{document}

\date{}

\title{\textbf{Closed subsets in Bishop topological groups}}
%\titlecomment{{\lsuper*}This paper is a major extension of~\cite{Pe17}.}

\author{Iosif Petrakis\\	%required
Mathematics Institute, Ludwig-Maximilians Universit\"{a}t M\"{u}nchen\\
petrakis@math.lmu.de}  %optional
%\thanks{thanks 1, optional.}	%optional

% \author[B.~Name2]{Bob Name2}	%optional
% \address{address2; addresses should initially be duplicated, even if
%   authors share an affiliation}	%optional
% \email{name2@email2; ditto for email addresses}  %optional
% \thanks{thanks 2, optional.}	%optional
% 
% \author[C.~Name3]{Carla Name3}	%optional
% \address{address 3}	%optional
% \urladdr{name3@url3\quad\rm{(optionally, a web-page can be specified)}}  %optional
% \thanks{thanks 3, optional.}	%optional

%% etc.

%% required for running head on odd and even pages, use suitable
%% abbreviations in case of long titles and many authors:

%%%%%%%%%%%%%%%%%%%%%%%%%%%%%%%%%%%%%%%%%%%%%%%%%%%%%%%%%%%%%%%%%%%%%%%%%%%

%% the abstract has to PRECEDE the command \maketitle:
%% be sure not to issue the \maketitle command twice!

%\keywords{Constructive topological algebra, Bishop topological groups, closed sets}

\maketitle

\begin{abstract}
We introduce the notion of a Bishop topological group i.e., a group $X$ equipped with a Bishop topology 
of functions $F$ such that the group operations of $X$ are Bishop morphisms with respect to $F$. A closed subset 
in the neighborhood structure of $X$ induced by its Bishop topology $F$ is defined in a positive way i.e., not as the
complement of an open subset in $X$. The corresponding closure operator, although it is not topological, in 
the classical sense, does not involve sequences. As countable choice (CC) is avoided, and in agreement
with Richman's critique on the use of CC in constructive mathematics, the fundamental facts on closed subsets in Bishop
topological groups shown here have a clear algorithmic content. We work within Bishop's informal
system of constructive mathematics $\BISH$, without countable choice, equipped with inductive definitions with 
rules of countably many premises.
\end{abstract}

\section{Introduction}
\label{sec: intro}

\noindent

The constructive non-viability of the notion of topological space is corroborated by the fact that 
many classical topological phenomena, like the duality between open and closed sets, are compatible only 
with classical logic. In a straightforward, constructive translation of general topology we cannot accept 
that the set-theoretic complement of a closed set is open. E.g., $\{0\}$ is a closed subset $\Real$, with 
respect to the topology on $\Real$ induced by its standard metric, while its complement cannot be accepted 
constructively as open, since that would imply the implication
$\neg(x = 0) \To (x > 0 \vee x < 0),$
which is (constructively) equivalent to the constructively unacceptable principle of Markov (see~\cite{BR87},
p.~15). The standard use of negative definitions in classical topology does not permit a smooth translation 
of classical topology to a constructive framework.
% \footnote{Maybe this is why the intuitionistic development of
% general topology (see~\cite{Tr66},~\cite{Gr81} and~\cite{Gr82}), which is mainly based on the use
% of intuitionistic logic in the study of the standard topological notions, is not under active
% current development.}.

In~\cite{Bi67}, chapter 3, Bishop defined a \textit{neighbourhood space} 
$\C N := (X, I, \nu)$, where $X, I$ are sets, and $(\nu_i)_{i \in I}$ is a family of subsets of $X$ indexed by $I$
(see~\cite{Pe20} for an elaborate study of this notion) that satisfies the following covering $(\NS_1)$ and 
neighborhood-condition $(\NS_2)$:\\[1mm]
$(\NS_1)$ $\bigcup_{i \in I}\nu(i) = X$.\\[1mm]
$(\NS_2)$ $\forall_{x \in X}\forall_{i, j \in I}\big[x \in \nu(i) \cap \nu(j) \To 
\exists_{k \in I}\big(x \in \nu(k) \ \& \  \nu(k) \subseteq \nu(i) \cap \nu(j)\big)\big]$. \\[1mm]
A subset $O$ of $X$ is called $\nu$-\textit{open}, if 
$\forall_{x \in O}\exists_{i \in I}\big(x \in \nu(i) \ \& \ \nu(i) \subseteq O\big)$.
An $\nu$-closed set $C$ is not defined negatively, as the complement of a $\nu$-open set, 
but \textit{positively} by the condition
$$\forall_{x \in X}\big(x \in \overline{C} \To x \in C\big),$$
$$x \in \overline{C}  :\TOT \forall_{i \in I}(x \in \nu(i) \To \nu(i) \between C),$$
where, if $A, B$ are subsets of $X$, then $A \between B : \TOT \exists_{y \in X}(y \in A \cap B)$.
If $(Y, J, \mu)$ is a
neighborhood space, a function $h : X \to Y$ is \textit{neighborhood-continuous}, if $h^{-1}(\mu(j))$ 
is $\nu$-open, for every $j \in J$. The concept of neighborhood space was proposed as a \textit{set-theoretic 
alternative} to the notion of topological space, and it is a formal topology in the sense of Sambin~\cite{Sa87},~\cite{Sa20}.

In~\cite{Bi67}, chapter 3, Bishop also defined the notion of \textit{function space} $\C F := (X, F)$,
where $X$ is a set and $F$ is a subset of $\D F(X)$, the real-valued functions on $X$, that satisfies 
the closure conditions of the set $\BR$ of Bishop-continuous functions from $\Real$ to $\Real$. 
Bishop called $F$ a \textit{topology} (of functions) on $X$. The set $\BR$ 
of Bishop-continuous functions $\phi : \Real \to \Real$ is the canonical topology of functions on $\Real$.
Bishop also defined inductively\footnote{This definition, together with the notion of the least algebra
of Borel sets generated by a family of complemented subsets of $X$, relative to a given set of real-valued
functions on $X$, are the main inductive definitions found in~\cite{Bi67}, both in chapter 3. The  notion 
of the least algebra of Borel sets is avoided in~\cite{BC72} and~\cite{BB85}, and the notion of the least 
topology is not developed neither in~\cite{Bi67} nor in~\cite{BB85}.} the \textit{least topology} of functions 
on $X$ that includes a given subset $F_0$ of $\D F(X)$. The concept of function space was proposed as a 
\textit{function-theoretic alternative} to the notion of topological space.

In~\cite{BB85}, p. 77, Bishop and Bridges expressed in a clear way the superiority of the 
function-theoretic notion of function space to the set-theoretic notion of neighborhood space.
% 
% \begin{quote}
%  Proximity is introduced into $X$ classically not by giving a family of functions, but by giving a 
%  family of subsets,
%  either open sets or neighborhoods. Classically, this is equivalent to giving a family of functions from 
%  $X$ to $\{0, 1\}$. Constructively, there is a vast difference: since functions are sharply defined, whereas 
%  most sets are fuzzy around the edges, only the all-too-rare detachable subsets of $X$ correspond to 
%  functions from
%  $X$ to $\{0, 1\}$. The fuzziness of sets is another reason to focus attention on function spaces instead of
%  neighborhood spaces.
% \end{quote}
% 
% 
% Since the notion of neighborhood-continuous function is the standard generalisation of pointwise continuity,
% it is not expected to reflect uniform continuity. E.g., as Palmgren remarkes in~\cite{Pa05}, the 
% composition of two neighborhood-continuous functions $f : \Real \to X$ and $g : X \to \Real$ need not 
% be in $\BR$.
As Bridges and Palmgren remark in~\cite{BP18}, ``little appears to have been done'' in the theory of 
neighborhood spaces. Ishihara has worked in~\cite{Is12} (and with co-authors in~\cite{IMSV06}) on their
connections to the apartness spaces of Bridges and V\^{\i}\c{t}\u{a} (see~\cite{BV11}), and in~\cite{Is13}
on their connections to Bishop's function spaces, while in~\cite{IP06} Ishihara and Palmgren studied the 
notion of quotient topology in neighborhood spaces.

Bridges talked on Bishop's function spaces at the first workshop on formal topology in 1997, and 
revived the subject of function spaces in~\cite{Br12}.  
Motivated by Bridges's paper, Ishihara showed in~\cite{Is13} the existence of an adjunction between the
category of neighbourhood spaces and the category of $\Phi$-closed pre-function spaces, where a 
pre-function space is an extension of the notion of a function space. In~\cite{Pe15a}-\cite{Pe20c} 
we try to develop the theory of function spaces, or \textit{Bishop spaces}, as we call them.
In~\cite{Pe19c} and in~\cite{Pe19b} we also study the applications of the theory of set-indexed families of 
Bishop sets in the theory of Bishop spaces. In~\cite{Ge18} 
connections between the theory of Bishop spaces and the theory of $C$-spaces of Escard\'o and Xu, developed
in~\cite{Xu15} and in~\cite{EX16}, are studied.

A group $X$ is a \textit{topological group}, if there is a topology of open sets $\C T$ on $X$ such that the corresponding operations
$+ \colon X \times X \to X$ and $- \colon X \to X$ are continuous functions with respect to $\C T$. The theory
of topological groups is very well-developed, with numerous applications 
(see~\cite{AT08},~\cite{Hi74} and~\cite{Wa93}). We call a group $X$, equipped with a Bishop topology 
of functions $F$, a \textit{Bishop topological group}, if the corresponding group operations
$+ \colon X \times X \to X$ and $- \colon X \to X$ are Bishop morphisms with respect to $F$. A \textit{Bishop morphism}
between Bishop spaces is the notion of arrow in the category of Bishop spaces that was introduced 
by Bridges in~\cite{Br12} and corresponds to the notion of a continuous function between topological spaces.

Most of the concepts of the theory of Bishop spaces are function-theoretic i.e., they are determined by the Bishop
topology of functions $F$ on $X$. Each Bishop topology $F$ generates a \textit{canonical neighborhood structure}, 
a family of basic open sets in $X$, described in section~\ref{sec: nbhd}. As explained above, a closed set $C$
with respect to this neighborhood structure is defined positively, and independently from its set-theoretic complement.
Generally we cannot show constructively that the set-theoretic complement $X \setminus C$ 
of a closed set $C$ is open. What we show in Theorem~\ref{thm: cr} though, is that a positive notion of complement, 
determined by $F$, the \textit{$F$-complement} $X \setminus_F C$ of $C$, is the largest open set included in $X \setminus C$.

In the main core of this paper we prove some fundamental properties of the closed sets in Bishop topological groups. 
Using functions to describe general properties of sets, and working with the aforementioned positive notion of closed set
gives us the opportunity to find constructive proofs with a clear computational content of results, which in 
many cases in the classical theory of topological groups depend on the use of classical negation. 
Moreover, our concepts and results avoid the use even of countable choice (CC). Although practicioners of Bishop-style
constructive mathematics usually embrace CC, avoiding it, and using non-sequential or non-choice-based arguments instead,
forces us to formulate ``better'' concepts and find ``better'' proofs. This standpoint was advocated first by Richman 
(see~\cite{Ri01} and~\cite{Sc04}). 

The study of closed sets in the neighborhood structure induced by the Bishop topology of a Bishop topological group shows
the fruitfulness of combining the two constructive proposals of Bishop to the classical topology of open sets. Moreover, the 
group-structure of a Bishop space $X$ helps us ``recover'' part of the classical duality between closed and open sets.
As Corollary~\ref{cor: corchar1} indicates, there are many cases of closed sets in a Bishop topological group
for which we can show that their set-theoretic complement is open!

The structure of this paper is the following:

\begin{itemize}
 \item In section~\ref{sec: basicdef} we include all definitions and facts on Bishop spaces that are necessary 
 to the rest of the paper. All proofs not given here are found in~\cite{Pe15}. For all results
 on constructive analysis that are used here without proof, we refer to~\cite{BB85}.
 \item In section~\ref{sec: nbhd} we give all definitions and results on the canonical neighborhood structure of a 
 Bishop topology that are used here. Theorem~\ref{thm: cr} is the result of this section that is most relevant to the study of
closed sets in Bishop topological groups. 
 \item In section~\ref{sec: btg} we introduce Bishop topological groups and we prove some of their fundamental properties.
 \item In section~\ref{sec: closed}, the central section of our paper, we prove fundamental properties of closed sets in Bishop
 topological groups. As we work with functions and positively defined concepts, avoiding the use of choice, our proofs
 generate clear algorithms. For all algebraic notions within $\BISH$ used here, we refer to~\cite{MRR88}.
% \item In section~\ref{sec: }
\end{itemize}

% positive
% 
% function-theoretic
% 
% no choice
% 
% better algorithms The results are not a surprise, although a constructive version of topological groups should 
% include them. The more interesting thing about them are the clear algorithms behind their proofs.
% 
% 
% 
% 
% 
% 
% fruitfulness of the Bishop-space and neighborhood space approach together smooth and fast applications to mathematics
% 
% The simple classical proofs of the following theorem with the use of nets rely on the axiom of choice. Our proofs are 
% computationally more interesting.

We work within Bishop's informal system of constructive mathematics $\BISH$, without countable choice, equipped
with inductive definitions with rules of countably many premises. 
A set-theoretic formal framework for this system\footnote{Extensional Martin-L\"of Type Theory or the theory of 
setoids within intensional Martin-L\"of Type Theory are possible type-theoretic systems for this informal system
(see~\cite{CDPS05}), although there choice, in the form of the distributivity of $\prod$ over $\sum$, is provable.} is
Myhill's $\CST^*$ without countable choice, or $\CZF$, equipped with a weak form of Aczel's regular extension axiom $\REA$
(see~\cite{AR10} and~\cite{LR03}).

\section{Fundamentals of Bishop spaces}
\label{sec: basicdef}

If $a, b \in \Real$, let $a \vee b := \max\{a, b\}$ and $a \wedge b := \min\{a, b\}$. Hence, $|a| = a \vee (-a)$.
If $f, g \in \D F(X)$, let $f =_{\D F(X)} g :\TOT \forall_{x \in X}\big(f(x) =_{\Real} f(y)\big)$. 
If $f, g \in \mathbb{F}(X)$, $\varepsilon > 0$ and $\Phi
\subseteq \mathbb{F}(X)$, let
$$U(g, f, \varepsilon) :\TOT \forall_{x \in X}\big(|g(x) - f(x)| \leq \varepsilon\big),$$ 
$$U(\Phi, f) :\TOT \forall_{\varepsilon > 0}\exists_{g \in \Phi}\big(U(g, f, \varepsilon)\big).$$
A set $X$ is \textit{inhabited}, if it has an element. We denote by $\overline{a}^X$, or simply by $a$, the constant 
function on $X$ with value $a \in \Real$, and by $\Const(X)$ their set.
%of all constant functions in the set $\D F(X)$ of real-valued functions on $X$.

% From now on $X$ is an inhabited set equipped with a Bishop topology $F$, $Y$ is an inhabited set equipped
% with a Bishop topology $G$ and $\C F := (X, F), \C G := (Y, G)$ are the corresponding Bishop spaces

\begin{definition}\label{def: bishop}
A Bishop space is a pair $\C F := (X, F)$, where $X$ is an inhabited set and $F$ is an \textit{extensional} subset
of $\D F(X)$ i.e., $\forall_{f,g \in \D F(X)}\big([f \in F \ \& \ g =_{\D F(X)} f] \To g \in F\big)$,
%$F \subseteq \D F(X)$,
such that the following conditions hold:\\[1mm]
%\footnote{The extensionality of $F$ is used in the proof 
%of Proposition~\ref{prp: morbaireone}(ii).}
$(\BS_1)$ $\Const(X) \subseteq F$.\\[1mm]
$(\BS_2)$ If $f, g \in F$, then $f + g \in F$.\\[1mm]
$(\BS_3)$ If $f \in F$ and $\phi \in \Bic(\Real)$, then $\phi \circ f \in F$.\\[1mm]
%, where $\Bic(\Real)$ is the set of all
%Bishop-continuous functions from $\Real$ to $\Real$ i.e., of all functions that are uniformly continuous on every
%closed interval $[-n, n]$, where $n \geq 1$.\\[1mm]
$(\BS_4)$ If $f \in \mathbb{F}(X)$ and $U(F, f)$, then $f \in F$.\\[1mm]
%$U(f, g_n, \frac{1}{n}) : \TOT 
%\forall_{x \in X}\big(|f(x) - g_n (x)| \leq \frac{1}{n}\big)$, for every $n \geq 1$, then $f \in F$.\\[1mm]
We call $F$ a Bishop topology on $X$. If $\C G := (Y, G)$ is a Bishop space, a Bishop morphism from $\C F$ to 
$\C G$ is a function $h : X \to Y$ such that $\forall_{g \in G}\big(g \circ h \in F\big)$. We denote by 
$\Mor(\C F, \C G)$ the set of Bishop morphisms from $\C F$ to $\C G$. If $h \in \Mor(\C F, \C G)$, we say that $h$ is 
\textit{open}, if $\forall_{f \in F}\exists_{g \in G}\big(f = g \circ h\big)$. If $h \in \Mor(\C F, \C G)$ is a bijection
and $h^{-1}$ is a Bishop morphism, we call $h$ a Bishop isomorphism.

\end{definition}

A Bishop morphism $h \in \Mor(\C F, \C G)$ is a ``continuous'' function from $\C F$ to $\C G$. If $h \in \Mor(\C F, \C G)$
is a bijection, then $h^{-1} \in \Mor(\C G, \C F)$ if and only if $h$ is open.
Let $\C R$ be the \textit{Bishop space of reals} $(\Real, \BR)$. It is easy to show that if $F$ is a topology on $X$, then 
$F = \Mor(\C F, \C R)$ i.e., an element of $F$ is a real-valued ``continuous'' function on $X$.
A Bishop topology $F$ on $X$ is an algebra and a lattice, where 
$f \vee g$ and $f \wedge g$ are defined pointwise, and $\Const(X) \subseteq F \subseteq \D F(X)$. If
$\D F^* (X)$ denotes the bounded elements of $\D F(X)$, then 
$F^* := F \cap \D F^*(X)$ is a Bishop topology on $X$. If $x =_X y$ is the given equality on $X$, a Bishop 
topology $F$ on $X$ \textit{separates the points} of $X$, or $F$ is \textit{separating} (see~\cite{Pe15a}), if 
%for their importance in the theory of Bishop spaces), if
$$\forall_{x, y \in X}\big[\forall_{f \in F}\big(f(x) =_{\Real} f(y)\big) \To x =_X y\big].$$
The \textit{canonical apartness relation} on $X$ induced by $F$ is defined by 
% In Proposition 5.1.3. of~\cite{Pe15} it is shown that $F$ separates the points of $X$ if and only if the induced 
% by $F$ apartness relation on $X$  
$$x \neq_F y :\TOT \exists_{f \in F}\big(f(x) \neq_{\Real} f(y)\big).$$

An apartness relation on $X$ is a positively defined inequality on $X$. E.g.,
if $a, b \in \Real$, then $a \neq_{\Real} b :\TOT |a - b| > 0$. In Proposition 5.1.2. of~\cite{Pe15} we show
that $a \neq_{\Real} b \TOT a \neq_{\BR} b$. 

\begin{definition}\label{def: base}
Turning the definitional clauses $(\BS_1)-(\BS_4)$ into inductive rules, 
the least topology $\VoF$ 
generated by a set $F_{0} \subseteq \mathbb{F}(X)$, called a subbase of 
$\VoF$, is defined by the following inductive rules:
 $$\frac{f_0 \in F_0}{f_0 \in \VoF}, \ \ \ \frac{f \in \VoF, \ g \in \D F(X), \ g =_{\D F(X)} f}{g \in \VoF},
\ \ \ \frac{a \in \Real}{\overline{a} \in \VoF}, \ \ \ \frac{f, g \in \VoF}{f + g \in \VoF},$$
 \begin{align*}
      \begin{aligned}
      \infer[]{f \in \VoF}{g_{1} \in \VoF \ \& \ 
      U(g_{1}, f, \frac{1}{2}), \ g_{2} \in \VoF \ \& \ 
      U(g_{2}, f, \frac{1}{2^{2}}), \ g_{3} \in \VoF \ \& \ U(g_{3}, f, \frac{1}{2^{3}}), \ldots}
      \end{aligned}.
  \end{align*}  
% $$\frac{f \in \VoF, \ \phi \in \BR}{\phi \circ f \in \VoF}, \ \ \ \frac{f \in \D F(X), \  
%  \big(g \in \VoF, \  \U(X; f, g, \varepsilon)\big)_{\vaRepsilon > 0}}{f \in \VoF},$$
% where the last rule is reduced to the following rule with countably many premisses
% $$\frac{f \in \D F(X), \ g_1 \in \VoF, \  \U(X; f, g_1, \frac{1}{2}), \ g_2 \in \VoF, \ 
% \U(X; f, g_2, \frac{1}{4}), \ldots }{f \in \VoF}.$$
The above rules induce the corresponding induction principle $\Ind_{\VoF}$ on $\VoF$. 
%A \textit{base} of $F$ is a subset $B$ of $F$ such that for every $f \in F$ there is a sequence 
%$(\beta_n)_{n = 1}^{\infty} \subseteq B$ such that $(\beta_n) \stackrel{u} \longrightarrow f$ on $X$.\\[1mm]
%$\forall_{n \geq 1}\big(U(f, g_n, \frac{1}{n})\big)$.
% If $A \subseteq X$, the \textit{relative} topology $F_{|A}$
% on $A$ has the set $\{f_{|A} \mid f \in F\}$ as a subbase.
\end{definition}

If $h \colon X \to Y$ and
$G = \bigvee G_0$, then one can show inductively i.e., with the use of $\Ind_{\bigvee G_0}$, that $h \in \Mor(\C F, \C G) \TOT 
\forall_{g_0 \in G_0}\big(g_0 \circ h \in F\big)$. We call this property the 
$\V$-\textit{lifting of morphisms}. 

\begin{definition}
If $\mathcal{F} = (X, F)$ and $\mathcal{G} = (Y, G)$ are given Bishop 
spaces, their \textit{product} is the structure $\mathcal{F} \times \mathcal{G} = 
(X \times Y, F \times G)$, where 
$$F \times G := \V \{f \circ \pi_{1} \mid f \in F\} \cup \{g \circ \pi_{2}
\mid g \in G\} := \V_{f \in F}^{g \in G} f \circ \pi_1, g \circ \pi_2,$$
and $\pi_{1}, \pi_{2}$ are the projections of $X \times Y$ to $X$ and $Y$, respectively. 
\end{definition}

It is straightforward to show that $\mathcal{F} \times \mathcal{G}$ satisfies the 
universal property for products and that $F \times G$ is the least topology which turns the 
projections $\pi_{1}, \pi_{2}$ into morphisms. If $F_{0}$ is a subbase of $F$ and $G_{0}$ is a subbase of 
$G$, then we show inductively that
$$\VoF \times \VoG = \V \{f_0 \circ \pi_{1} \mid f_0 \in F_0\} \cup \{g_0 \circ \pi_{2}
\mid g_0 \in G_0\} := \V_{f_0 \in F_0}^{g_0 \in G_0} f_0 \circ \pi_1, g_0 \circ \pi_2.$$
Consequently,  $\Bic(\mathbb{R}) \times \Bic(\mathbb{R}) = 
\V \id_{\mathbb{R}} \circ \pi_{1}, \id_{\mathbb{R}} \circ \pi_{2} = \V \pi_{1}, \pi_{2}$.
% The arbitrary product $\prod_{i \in I}\Least_{i}$ of a family $(\Least_{i})_{i \in I}$ of Bishop spaces
% indexed by some inhabited set $I$ is defined similarly. Using the $\V$-lifting of morphisms it is easy
% to show the following proposition.

% 
% 
% \begin{proposition}\label{prp: newprod} Let $\mathcal{F} = (X, F)$, $\mathcal{G} = (Y, G)$, 
% $\mathcal{H} = (Z, H)$ be Bishop spaces, $A \subseteq X$, and $B \subseteq Y$.\\[1mm]
% $(i)$ $j \in \Mor(\mathcal{H}, \mathcal{F} \times \mathcal{G})$ if
% and only if $\pi_{1} \circ j \in \Mor(\mathcal{H}, \mathcal{F})$ and $\pi_{2} \circ j \in
% \Mor(\mathcal{H}, \mathcal{G})$.\\[1mm]
% % $(ii)$ If $e: X \rightarrow B$, then $e \in \Mor(\mathcal{F}, \mathcal{G}) \TOT
% % e \in \Mor(\mathcal{F}, \mathcal{G}_{|B})$.\\[1mm] 
% % $(iii)$ $(F \times G)_{|A \times B} = F_{|A} \times G_{|B}$.
% \end{proposition}

%Note that Proposition~\ref{prp: newprod}(i) and (iii) hold for arbitrary products too.

\begin{corollary}\label{crl: corprod3} Let $\mathcal{H} = (Z, H), \mathcal{F} = (X, F), \mathcal{G} = (Y, G)$
be Bishop spaces.\\[1mm]
\normalfont (i)
\itshape If $h_{1} : Z \rightarrow X$, $h_{2} : Z \rightarrow Y$, the map 
$h_1 \times h_2 : Z \rightarrow X \times Y,$ defined by
$z \mapsto (h_{1}(z), h_{2}(z)),$
is in $\Mor(\mathcal{H}, \mathcal{F} \times \mathcal{G})$
if and only if $h_{1} \in \Mor(\mathcal{H}, \mathcal{F})$ and $h_{2} \in \Mor(\mathcal{H}, \mathcal{G})$.\\[1mm]
\normalfont (ii)
\itshape
If $e_{1} : X \rightarrow Z$, $e_{2} : Y \rightarrow Z$, then the map
$e_1 \otimes e_2 : X \times Y \rightarrow Z \times Z,$ defined by $(x, y) \mapsto (e_{1}(x), e_{2}(y)),$
is in $\Mor(\mathcal{F} \times \mathcal{G}, \C H \times \C H)$
if and only if $e_{1} \in \Mor(\mathcal{F}, \mathcal{H})$ and $e_{2} \in \Mor(\mathcal{G}, \mathcal{H})$. 
%Moreover, if $h_1, h_2$ are open set-epimorphisms, then $h_1 \otimes h_2$ is an open set-epimorphism.
\end{corollary}

% 
% 
% \begin{proposition}\label{prp: gcorcorprod3} Suppose that $\mathcal{F}_{n} = (X_{n}, F_{n})$ and
%  $\mathcal{G}_{n} = (Y_{n}, F_{n})$ are two sequences of Bishop spaces and $h_{n}: X_{n} \rightarrow Y_{n}$ 
%  are given functions. Then the function 
%  $$\prod_{n}f_{n} : \prod_{n}X_{n} \rightarrow \prod_{n}Y_{n},$$ 
% $$(x_{n})_{n} \mapsto (f_{n}(x_{n}))_{n},$$
% is in $\Mor(\prod_{n}\mathcal{F}_{n}, \prod_{n}\mathcal{G}_{n})$ 
% if and only if $f_{n} \in \Mor(\mathcal{F}_{n}, \mathcal{G}_{n})$, for every $n$. Moreover, if every 
% $f_{n}$ is an open set-epimorphism, then $\prod_{n}f_{n}$ is an open set-epimorphism.
% \end{proposition}
% 
% 
% 
% \begin{proposition}\label{prp: prod4} Suppose that $\mathcal{F} = (X, F)$ and $\mathcal{G} = (Y, G)$ are Bishop 
% spaces, $A \subseteq X$ and $B \subseteq Y$.\\
% \normalfont (i)
% \itshape If $F = F_{b}$ and $G = G_{b}$, then $F \times G = (F \times G)_{b}$.\\
% \normalfont (ii)
% \itshape If $F = F(A)$ and $G = G(B)$, then $F \times G = (F \times G)(A \times B)$.
% 
% \end{proposition}
% 
% 

\begin{proposition}\label{prp: prod6} Suppose that $\mathcal{F} = (X, F)$, $\mathcal{G} = (Y, G)$, $\mathcal{H} = (Z, H)$
are Bishop spaces, $x \in X, y \in Y$, $\phi: X \times Y \rightarrow \mathbb{R} \in F \times G$ and
$\Phi: X \times Y \rightarrow Z \in \Mor(\mathcal{F} \times \mathcal{G}, \mathcal{H})$.\\
\normalfont (i)
\itshape $i_{x}: Y \rightarrow X \times Y$, $y \mapsto (x, y)$, and $i_{y}: X \rightarrow X \times Y$,
$x \mapsto (x, y)$, are open morphisms.\\
\normalfont (ii)
\itshape $\phi_{x}: Y \rightarrow \mathbb{R}$, $y \mapsto \phi(x, y)$, and $\phi_{y}: X \rightarrow 
\mathbb{R}$, $x \mapsto \phi(x, y)$, are in $G$ and $F$, respectively.\\
\normalfont (iii)
\itshape $\Phi_{x}: Y \rightarrow Z$, $y \mapsto \Phi(x, y)$, and $\Phi_{y}: X \rightarrow Z$, $x \mapsto \Phi(x, y)$, 
are in $\Mor(\mathcal{G}, \mathcal{H})$ and $\Mor(\mathcal{F}, \mathcal{H})$, respectively.
\end{proposition}

\begin{proof}
(i) We show it only for $i_{y}$. By the $\Least$-lifting of morphisms we have that
$i_{y} \in \Mor(\mathcal{F}, \mathcal{F} \times \mathcal{G}) \TOT \forall_{f \in F}((f \circ \pi_{1}) 
 \circ i_{y} \in F) \ \& \  \forall_{g \in G}((g \circ \pi_{2}) \circ i_{y} \in F)$.
If $f \in F$, then $(f \circ \pi_{1}) \circ i_{y} = f$, which shows also that $i_{y}$ is open, while if
$g \in G$, then $(g \circ \pi_{2}) \circ i_{y} = \overline{g(y)} \in F$.\\
(ii) We show it only for $\phi_{y}$. We have that $\phi_{y} = \phi \circ i_{y}$, since $(\phi \circ i_{y})(x) = 
\phi(x, y) = \phi_{y}(x)$, for each $x \in X$. Since $i_{y} \in \Mor(\mathcal{F}, \mathcal{F} \times \mathcal{G})$
and $\phi \in F \times G$, we get that $\phi \circ i_{y} = \phi_{y} \in F$.\\
(iii) The proof is similar to the proof of (ii), Actually, (ii) is a special case of (iii).
\end{proof}

\section{The neighborhood structure of a Bishop topology}
\label{sec: nbhd}

If $F$ is a Bishop topology on $X$, the neighborhood structure on $X$ induced by $F$ is the family $(U(f)_{f \in F}$, where
$U(f) := \{x \in X \mid f(x) > 0\}$. The covering condition $(\NS_1)$ follows from the equality
$U(\overline{\alpha}^X) = X$, where $a > 0$, and the neighborhood-condition $(\NS_2)$ follows from the equality 
$U(f) \cap U(g) = U(f \wedge g)$, for every $f, g \in F$. Consequently, $O \subseteq X$ is \textit{open} if
$$\forall_{x \in O}\exists_{f \in F}\big(f(x) > 0 \ \& \ U(f) \subseteq O\big),$$
and $C \subseteq X$ is \textit{closed}, if $\forall_{x \in X}\big(x \in \overline{C} \To x \in C\big)$, where
$$x \in \overline{C}  :\TOT \forall_{f \in F}\big(f(x) > 0  \To \exists_{c \in C}\big(f(c) > 0\big)\big).$$

\begin{proposition}\label{prp: top1} Let $\mathcal{F} = (X, F)$, $\mathcal{G} = (Y, G)$ 
be Bishop spaces, $f \in F$ and $h: X \rightarrow Y$.\\[1mm]
\normalfont (i)
\itshape $\mathcal{O}$ is open in $N(\Bic(\Real))$ if and only if it is open in the standard topology on $\Real$.\\[1mm]
% \normalfont (ii)
% \itshape 
% The induced topology $\mathcal{T}_{\mathcal{N}(F)}$ is the smallest topology $\mathcal{T}$ of opens on $X$ such that
% $f \in C(X)$, for every $f \in F$.\\
\normalfont (ii)
\itshape If $h \in \Mor(\mathcal{F}, \mathcal{G})$, then $h$ is neighborhood-continuous.\\[1mm]
\normalfont (iii)
\itshape If $h \in \Mor(\mathcal{F}, \mathcal{G})$, the inverse image of a closed set in $Y$ under $h$ is closed 
in $X$.\\[1mm]
\normalfont (iv)
\itshape If $h \in \Mor(\mathcal{F}, \mathcal{G})$ and $A \subseteq X$, then $h(\overline{A}) \subseteq
\overline{h(A)}$.\\[1mm]
\normalfont (v)
\itshape The set-theoretic complement $X \setminus U(f)$ of $U(f)$ in $X$ is closed in $X$, for every $f \in F$.\\[1mm]
\normalfont (vi)
\itshape The zero set $\zeta(f) := \{x \in X \mid f(x) = 0\}$ is closed. 
\end{proposition}

\begin{proof}
For (i)-(v) see~\cite{Pe15}, Proposition 4.4, and for (vi) see~\cite{Pe15}, Proposition 5.3.2.
\end{proof}

Clearly, $C \subseteq X$ is closed if and only if $C = \overline{C}$, and the intersection of closed sets is closed.
The closure operator $A \mapsto \overline{A}$ is not topological, in the classical sense, as we cannot show constructively
that the union of two closed sets is closed, in general,closed (see also~\cite{BB85}, p.~79). 
If $A, B \subseteq X$ and $F$ a Bishop topology on $X$, then it is 
straightforward to show that (i) $A \subseteq \overline{A}$, (ii) $\overline{\overline{A}} \subseteq \overline{A}$, (iii) 
$A \subseteq B \To \overline{A} \subseteq \overline{B}$, and (iv) 
$\overline{A} \cup \overline{B} \subseteq \overline{A \cup B}$. The inverse inclusion 
$\overline{A \cup B} \subseteq \overline{A} \cup \overline{B}$ cannot be shown constructively.
% 
% \begin{proposition}\label{prp: top2}
% Let $A, B \subseteq X$ and $F$ a Bishop topology on $X$.\\[1mm]
% \normalfont (i)
% \itshape $A \subseteq \overline{A}$.\\[1mm]
% \normalfont (ii)
% \itshape
% $\overline{\overline{A}} \subseteq \overline{A}$.\\[1mm]
% \normalfont (iii)
% \itshape
% $A \subseteq B \To \overline{A} \subseteq \overline{B}$.\\[1mm]
% \normalfont (iv)
% \itshape
% $\overline{A} \cup \overline{B} \subseteq \overline{A \cup B}$.
% \end{proposition}
If $F$ is a Bishop topology on $X$ and $C \subseteq X$, we define positively and through $F$ a stronger relation 
``$x$ is not in $C$'', where $x \in X$, by
$$x \notin_F C :\TOT \exists_{f \in F}\big(f \colon x \notin_F C\big),$$
$$f \colon x \notin_F C :\TOT f(x) > 0 \ \& \ \forall_{c \in C}\big(f(c) = 0\big).$$
Classically one can show that $F$ is always completely regular i.e., if $C$ is closed 
in $X$ and $x \notin C$, then $x \notin_F C$ (see~\cite{Pe15}, Proposition 3.7.6). Constructively we can show the following. 

\begin{theorem}\label{thm: cr}
Let $F$ be a Bishop topology on $X$ and $C$ closed in $X$. The $F$-complement  
$X \setminus_F C := \{x \in X \mid \exists_{f \in F}\big(f \colon x \notin_F C\}$ of $C$ in $X$
is the largest open set included in $X \setminus C$.  
\end{theorem}

\begin{proof}
We show that $X \setminus_F C$ is an open set included in $X \setminus C$ 
and if $O$ is an open set included in $X \setminus C$, then $O \subseteq X \setminus_F C$.
First we show that $X \setminus_F C$ is open. If $x \in X \setminus_F C$, let $f \in F$ such that 
$f \colon x \notin_F C$. Clearly, $x \in U(f)$, and if $y \in X$ such that $y \in U(f)$, then $f \colon y \notin_F C$ 
i.e., $y \in X \setminus_F C$. Next we show that if 
$x \in X \setminus_F C$, then $x \in X \setminus C$. If $x \in C$, then $f(x) > 0$ and $f(x) = 0$, which is a contradiction. 
Suppose next that $O$ is an open set included in $X \setminus C$. We show that if $x \in O$, then $x \in X \setminus_F C$.
Since $O$ is open, there is $g \in F$ such that $g(x) > 0$ and $U(g) \subseteq O \subseteq X \setminus C$.
We have that $g(c) \leq 0$, for every $c \in C$, since if $g(c) > 0$, for some $c \in C$, then $c \in U(g)$, hence
$c \in X \setminus C$, which is a contradiction. By  the constructively valid implication $\neg(a > 0) \To a \leq 0$,
for every $a \in \Real$ (see~\cite{BB85}, Lemma 2.18), we get $g(c) \leq 0$. Since $g \vee \overline{0}^X \in F$, we
conclude that $g \vee \overline{0}^X \colon x \notin_F C$, hence $x \in X \setminus_F C$.
\end{proof}

Although we cannot show in general that $X \setminus C$ is open, and hence $X \setminus C = X \setminus_F C$,
we can replace this computationally dubious result by the computationally meaningful fact that 
$X \setminus_F C$ is the largest open set included in $X \setminus C$. The next result is used in the proofs 
of Theorem~\ref{thm: bnbh1}(ii) and Theorem~\ref{thm: bnbh2}(ii).

\begin{proposition}\label{prp: tight}
If $F$ is a Bishop topology on $X$, the following are equivalent:\\[1mm]
\normalfont (i)
\itshape $F$ separates the points of $X$.\\[1mm]
\normalfont (ii)
\itshape The inequality $\neq_F$ generated by $F$ is tight i.e., $\neg(x \neq_F y) \To x =_X y$, 
for every $x, y \in X$.\\[1mm]
\normalfont (iii)
\itshape The singleton $\{x\}$ is closed, for every $x \in X$.
\end{proposition}

\begin{proof}
(i)$\To$(ii) Let $\neg (x =_F y) :\TOT \neg [\exists_{f \in F}\big(f(x) \neq_{\Real} f(y)\big)]$, for some $x, y \in X$.
We show that $\forall_{f \in F}\big(f(x) =_{\Real} f(y)\big)$. Let $f \in F$ such that $f(x) \neq_{\Real} f(y)$. By our 
hypothesis on $x, y$ this is impossible, hence by the tightness of $\neq_{\Real}$ we conclude that $f(x) =_{\Real} f(y)$.\\
(ii)$\To$(iii) Let $x, y \in X$ such that $\forall_{f \in F}\big(f(y) > 0 \To f(x) > 0\big)$. We show that $y =_X x$, 
by showing that $\neg\big(y \neq_F x\big)$. Suppose that $y \neq_F x$ and, without loss of generality,
let $g \in F$ such that $g(y) = 1$ and $g(x) = 0$. By the hypothesis on $x$ we have that $g(y) > 0 \To g(x) > 0$,
and we get the required contradiction.\\
(iii)$\To$(i) Let $x, y \in X$ such that $\forall_{f \in F}\big(f(x) =_{\Real} f(y)\big)$. We show that 
$y \in \overline{\{x\}}$, hence $y = x$. Let $f \in F$ such that $f(y) > 0$. Since $f(x) = f(y)$, we get $f(x) > 0$.
\end{proof}

% We use the tightness of $\neq_F$, where $F$ separates the points of $X$, 
% in the proofs of Theorem~\ref{thm: bnbh1}(ii) and Theorem~\ref{thm: bnbh2}(ii).

\begin{proposition}\label{prp: closed3}
 If $C$ is closed in $X$ and $D \subseteq Y$ is closed in $Y$, 
 %then
 $C \times D$ is closed in $X \times Y$.
\end{proposition}

\begin{proof}
 Let $(x,y) \in \overline{C \times D}$ i.e., if $h(x,y) > 0$, there is $(u,w) \in C \times D$ such that $h(u,w) > 0$, 
 for every $h \in F \times G$. We show that $x \in \overline{C}$ and (similarly) $y \in \overline{D}$,
 hence $(x,y) \in C \times D$. Let $f \in F$ such that $f(x) > 0$. Since $(f \circ \pi_1)(x, y) > 0$ and $f \circ \pi_1 
 \in F \times G$, there is $(u,w) \in C \times D$ such that $(f \circ \pi_1)(u,w) := f(u) > 0$, hence there is $u \in C$ 
 with $f(u) > 0$.
 \end{proof}

\section{Bishop topological groups}
\label{sec: btg}

\begin{definition}\label{def: bgroup}
A Bishop topological group is a structure $\B {\C F} := (X, +, 0, - ; F)$, where $\C X := (X, +, 0, -)$
is a group and $\C F:= (X, F)$ is 
a Bishop space such that $+ : X \times X \to X \in \Mor(\C F \times \C F, \C F)$ and 
$- : X \to X \in \Mor(\C F, \C F)$. If necessary, we also use the notations $+^{\mathsmaller{X}}, 0^{\mathsmaller{X}}$,
and $-^{\mathsmaller{X}}$ for the operations of the group $\C X$. If $f \in F$, let 
$f_+ := f \circ +$ and $f_- := f \circ -$.
% We say that $f$ is \textit{even}, if $f_- = f$, and $f$ is \textit{odd}, if $f_- = -f$. Let $\Even(F)$ be the set of even
% elements of $F$, and $\Odd(F)$ the set of odd elements of $F$.
\end{definition}

By the definition of a Bishop morphism we get
$$+ \in \Mor(\C F \times \C F, \C F) :\TOT \forall_{f \in F}\big(f \circ + \in F \times F\big) 
:\TOT \forall_{f \in F}\big(f_+ \in F \times F\big),$$
$$- \in \Mor(\C F, \C F) :\TOT \forall_{f \in F}\big(f \circ - \in F \big) 
:\TOT \forall_{f \in F}\big(f_{-} \in F \big).$$

\begin{example}[The additive group of reals]\label{ex: realsplus}
\normalfont
The structure $\B {\C R} := (\Real, +, 0, -; \BR)$ is a Bishop topological group.
By the $\V$-lifting of morphisms $+ \in \Mor(\C R \times \C R, \C R) \TOT \id_{\Real} \circ + \in \BR \times \BR.$
If $x, y \in \Real$, then
$\big(\id_{\Real} \circ +\big)(x, y) := x + y := (\pi_1 + \pi_2)(x, y)$
i.e., $\id_{\Real} \circ + = \pi_1 + \pi_2 \in \BR \times \BR = \V \pi_1, \pi_2$. Similarly, 
$- \in \Mor(\C R, \C R) \TOT \id_{\Real} \circ - \in \BR$. If $x \in \Real$, then 
$\big(\id_{\Real} \circ -\big)(x) := -x := - \id_{\Real}(x)$
i.e., $\id_{\Real} \circ - =  - \id_{\Real} \in \BR$.
\end{example}

\begin{example}[The trivial Bishop topological group]\label{ex: trivial}
If $\C X := (X, +, 0, -)$ is a group, then $\Const(X)$ is the \textit{trivial} Bishop topology on $X$. 
If $a \in \Real$, then $\overline{a}^X \circ + = \overline{a}^{X \times X} \in \Const(X \times X)
= \Const(X) \times \Const(X)$, and $\overline{a}^X \circ - = \overline{a}^{X} \in \Const(X)$.
\end{example}

Unless otherwise stated, from now on, $X, Y$ \textit{are Bishop topological groups with  
$F$ and $G$ Bishop topologies on $X$ and $Y$, respectively}.

\begin{proposition}\label{prp: bgroup3}
%Let $\B {\C F} := (X, +, 0, - ; F)$ be a Bishop topological group.\\[1mm]
\normalfont (i)
\itshape The function $- : X \to X$ is a Bishop isomorphism.\\[1mm]
%$(ii)$ For every $f \in F$ there is a symmetric function $g$ such that $g \leq f$.\\[1mm]
\normalfont (ii)
\itshape For every $x_0 \in X$ the functions $+_{x_0}^1, +_{x_0}^2 : X \to X$, defined by  
$+_{x_0}^1 (x) := x_0 + x$ and $+_{x_0}^2(x) := x + x_0,$
for every $x \in X$, are Bishop morphisms.\\[1mm]
\normalfont (iii)
\itshape If $f \in F$, the functions $f_{x_0}^1, f_{x_0}^2 : X \to \Real$, defined by
$f_{x_0}^1 (x) := f(x_0 + x)$ and $f_{x_0}^2(x) := f(x + x_0),$
for every $x \in X$, are in $F$.\\[1mm]
\normalfont (iv)
\itshape For every $x_0 \in X$ the functions $+_{x_0}^1, +_{x_0}^2 : X \to X$ are Bishop isomorphisms.
\end{proposition}

\begin{proof}
(i) By definition $- \in \Mor(\C F, \C F)$, and it is a bijection. It is also open i.e., $\forall_{f \in F}\exists_{g \in F}
\big(f = g_-\big)$. If $f \in F$, we have that $f = (f_-)_-$ and $f_- \in F$.\\
(ii) and (iii) If $i_{x_0}^1 : X \to X \times X$ is defined by $i_{x_0}^1(x) := (x_0, x)$, for every $x \in X$, then 
$i_{x_0}^1 \in \Mor(\C F, \C F \times \C F)$ and $+_{x_0}^1 := + \circ i_{x_0}^1 \in \Mor(\C F, \C F)$ as a composition of 
Bishop morphisms
\begin{center}
\begin{tikzpicture}

\node (E) at (0,0) {$X$};
\node[right=of E] (F) {$X \times X$};
\node[right=of F] (A) {$X$};
\node[right=of A] (B) {$\Real$.};

\draw[->] (E)--(F) node [midway,above] {$\ \ \ \ \ \mathsmaller{i_{x_0}^1}$};
\draw[->] (F)--(A) node [midway,above] {$+$};
\draw[->] (A)--(B) node [midway,above] {$\mathsmaller{f} \ \ $};
\draw[->,bend right] (E) to node [midway,below] {$+_{x_0}^1$} (A) ;
\draw[->,bend left] (E) to node [midway,above] {$f_{x_0}^1$} (B) ;

\end{tikzpicture}
\end{center}
$f_{x_0}^1 = f \circ +_{x_0}^1 \in F$, as a composition of Bishop morphisms. For $+_{x_0}^2$ we work similarly.\\
(iv) Clearly, $+_{x_0}^1$ is a bijection. It is also open, since for every $f \in F$ we have that 
$f = f_{-x_0}^1 \circ +_{x_0}^1,$
and by (iv) $f_{-x_0}^1 \in F$. For $f_{x_0}^2$ we work similarly.
\end{proof}

\begin{proposition}\label{prp: bgroup4}
%Let $\B {\C F} := (X, +, 0, - ; F)$ be a Bishop topological group. 
The function $k : X \times X \to X \times X$, defined by
$k(x, y) := (x, -y)$, for every $(x, y) \in X \times X$, is a Bishop isomorphism.
\end{proposition}

\begin{proof}
Clearly, $k$ is a bijection. By the $\V$-lifting of morphisms we have that
$k \in \Mor(\C F \times \C F, \C F \times \C F)$ if and only if $(f \circ \pi_1) \circ k \in F \times F$ and
$(f \circ \pi_2) \circ k \in F \times F\big)$, for every $f \in F$.
Let $f \in F$ and $(x, y) \in X \times X$. Since
$[(f \circ \pi_1) \circ k](x, y) := (f \circ \pi_1)(x, -y) := f(x) := (f \circ \pi_1)(x, y),$
we get $(f \circ \pi_1) \circ k = f \circ \pi_1 \in F \times F$. Moreover, 
$[(f \circ \pi_2) \circ k](x, y) := (f \circ \pi_2)(x, -y) := f(-y) := f_-(y) := (f_- \circ \pi_2)(x, y),$
i.e., $(f \circ \pi_2) \circ k = (f_- \circ \pi_2) \in F \times F$, since $f_- \in F$. Since $(k \circ k)(x, y) := 
(x, -(-y)) = (x, y)$, $k$ is its own inverse, hence $k$ is a Bishop isomorphism.
\end{proof}

\begin{proposition}\label{prp: bgroup5}
Let $\C X := (X, +, 0, -)$ be a group, $F$ a Bishop topology on $X$, and $\sub : X \times X \to
X$ be defined by $\sub(x, y) := x - y,$ for every $(x, y) \in X \times X$. Then 
$\B {\C F} := (X, +, 0, - ; F)$ is a Bishop topological group if and only if 
$\sub \in \Mor(\C F \times \C F, \C F)$.
\end{proposition}

\begin{proof}
If $\B {\C F}$ is a Bishop topological group, then $\sub = + \circ k \in \Mor(\C F \times \C F, \C F)$
\begin{center}
\begin{tikzpicture}

\node (E) at (0,0) {$X \times X$};
\node[right=of E] (F) {$X \times X$};
\node[above=of F] (A) {};
\node[right=of A] (B) {$X$};
\node[below=of F] (C) {};
\node[right=of C] (D) {$X$.};

\draw[->] (E)--(F) node [midway,above] {$k$};
\draw[->] (F)--(B) node [midway,left] {$+$};
\draw[->] (F)--(D) node [midway,right] {$\sub$};
\draw[->,bend left] (E) to node [midway,above] {$\sub$} (B) ;
\draw[->,bend right] (E) to node [midway,below] {$+$} (D) ;

\end{tikzpicture}
\end{center}
For the converse, notice that $+ = \sub \circ k \in \Mor(\C F \times \C F, \C F)$.
%as a composition of Bishop morphisms. 
By Proposition~\ref{prp: prod6}(iii) we get $- = \sub_0 \in \Mor(\C F, \C F)$, where $\sub_0(y) := \sub(0, y) :=
-y$, for every $y \in X$.
\end{proof}

\begin{proposition}\label{prp: bgroup6}
Let $\B {\C F} := (X, +, 0, - ; F)$ be a Bishop topological group and $\C G := (Y, G)$ a Bishop space.
Let the functions $+^{\mathsmaller{\mathsmaller{\to}}} : \Mor(\C G, \C F) \times \Mor(\C G, \C F) \to \Mor(\C G, \C F)$, 
$-^{\mathsmaller{\mathsmaller{\to}}} : \Mor(\C G, \C F) \to \Mor(\C G, \C F)$, and $0^{\mathsmaller{\to}} : Y \to X$
be defined by
$(h_1 +^{\mathsmaller{\to}} h_2)(y) := h_1(y) + h_2(y)$, $(-^{\mathsmaller{\to}}h)(y) := -h(y)$, and
$0^{\mathsmaller{\to}}(y) := 0$, for every $y \in Y$, and $h_1, h_2, h \in \Mor(\C G, \C F)$.
Then $\big(\Mor(\C G, \C F), +^{\mathsmaller{\to}}, 0^{\mathsmaller{\to}}, -^{\mathsmaller{\to}}\big)$
is a group.
\end{proposition}

One can show that the group
$\big(\Mor(\C G, \C F), +^{\mathsmaller{\to}}, 0^{\mathsmaller{\to}}, -^{\mathsmaller{\to}}\big)$,
equipped with the pointwise exponential Bishop topology (see~\cite{Pe15}, section 4.3), is a Bishop topological group.
Notice that if $Y$ has also a group structure compatible
with $G$, and if $h_1, h_2, h$ are group homomorphisms, then $h_1 +^{\mathsmaller{\to}} h_2, -^{\mathsmaller{\to}}h$ and
$0^{\mathsmaller{\to}}$ are also group homomorphisms.

\begin{definition}\label{def: BTopGrp}
Let $\B {\C F} := (X, +^{\mathsmaller{X}}, 0^{\mathsmaller{X}}, -^{\mathsmaller{X}} ; F)$ and 
$\B {\C G} := (Y, +^{\mathsmaller{Y}}, 0^{\mathsmaller{Y}}, -^{\mathsmaller{Y}} ; G)$ be Bishop topological groups. If
$h \in \Mor(\C G, \C F)$ such that $h$ is a $(\C X, \C Y)$-group homomorphism, then we call $h$ a \textit{Bishop group
homomorphism}, or simpler, a \textit{Bishop homomorphism}. We denote by $\MOR(\B {\C F}, \B {\C G})$ the set of all Bishop
homomorphisms from $\B {\C F}$ to $\B {\C G}$. Let $\BTopGrp$ be the category of Bishop topological
groups with Bishop group homomorphisms. 
\end{definition}

\begin{proposition}\label{prp: homR}
Let $a \in \Real$ and let $h_a : \Real \to \Real$ be defined by $h_a(x) = ax$, for every $x \in \Real$. Then 
$h_a \in \MOR(\B {\C R}, \B {\C R})$. Conversely, if $h \in \MOR(\B {\C R}, \B {\C R})$, there is $a \in \Real$ such that 
$h = h_a$.
\end{proposition}

\begin{proof}
$h_a$ is a group homomorphism, and by the $\V$-lifting of morphisms  
$h_a \in \Mor(\C \Real, \C \Real) \TOT \id_{\Real} \circ h_a := h_a \in \BR$, which holds, since 
$h_a := a \cdot \id_{\Real} \in \BR$. If $h \in \MOR(\B {\C R}, \B {\C R})$, let its restriction $h_{|\D Q}$, where
$h_{|\D Q} : \D Q \to \D Q$ is a group homomorphism given by $\big(h_{|\D Q}\big)(q) = h(1) q$, for every $q \in \D Q$.
Since $\D Q$ is metrically dense in $\Real$, by Proposition 4.7.15. in~\cite{Pe15} we have that 
$\Bic(\D Q) = \BR_{|\D Q} = \{\phi_{|\D Q} \mid \phi \in \BR\}.$
Hence $h_{|\D Q} \in \Bic(\D Q)$. By Lemma 4.7.13. in~\cite{Pe15} there is a unique (up to equality) extension of 
$h_{|\D Q}$ in $\BR$. Hence $h = h_{h(1)}$.
\end{proof}

\begin{proposition}\label{prp: bproduct}
Let $\B {\C F} := (X, +^{\mathsmaller{X}}, 0^{\mathsmaller{X}}, -^{\mathsmaller{X}} ; F)$ and
$\B {\C G} := (Y, +^{\mathsmaller{Y}}, 0^{\mathsmaller{Y}}, -^{\mathsmaller{Y}} ; G)$ be Bishop topological groups.
Let the functions $+^{\mathsmaller{\mathsmaller{X \times Y}}} : (X \times Y) \times (X \times Y) \to X \times Y$, 
$-^{\mathsmaller{\mathsmaller{X \times Y}}} : X \times Y \to X \times Y$, and 
$0^{\mathsmaller{X \times Y}} \in X \times Y$ be defined by
$(x, y) +^{\mathsmaller{X \times Y}} (x{'}, y{'}) := \big(x +^{\mathsmaller{X}} x{'}, y +^{\mathsmaller{Y}} y{'}\big),$
$-^{\mathsmaller{X \times Y}}(x, y) := \big(-^{\mathsmaller{X}} x, -^{\mathsmaller{Y}}y\big),$ and
$0^{\mathsmaller{X \times Y}} := \big(0^{\mathsmaller{X}}, 0^{\mathsmaller{Y}}\big)$. 
Then $\B {\C F \times \C G} := (X \times Y, +^{\mathsmaller{X \times Y}}, 0^{\mathsmaller{X \times Y}},
-^{\mathsmaller{X \times Y}} ; F \times G)$ is a Bishop topological group.
\end{proposition}

\begin{proof} 
The proof that $X \times Y$ is a group is omitted as trivial. By the $\V$-lifting of morphisms
$+^{\mathsmaller{X \times Y}} \in \Mor\big([\C F \times \C G] \times [\C F \times \C G], \C F \times \C G\big)$ 
if and only if
$$
\forall_{f \in F}\forall_{g \in G}\bigg(\big(f \circ \pi^{\mathsmaller{X}}\big) \circ +^{\mathsmaller{X \times Y}} \in
[F \times G] \times [F \times G] \& \ \big(g \circ \pi^{\mathsmaller{Y}}\big) \circ +^{\mathsmaller{X \times Y}} 
\in [F \times G] \times [F \times G]\bigg),$$
where by the $\V$-lifting of the product Bishop topology
$$[F \times G] \times [F \times G] = \V_{f \in F}^{g \in G} 
\big(f \circ \pi^{\mathsmaller{X}}\big) \circ \pi_1^{\mathsmaller{X \times Y}}, 
\big(g \circ \pi^{\mathsmaller{Y}}\big) \circ \pi_1^{\mathsmaller{X \times Y}},
\big(f \circ \pi^{\mathsmaller{X}}\big) \circ \pi_2^{\mathsmaller{X \times Y}},
\big(g \circ \pi^{\mathsmaller{Y}}\big) \circ \pi_2^{\mathsmaller{X \times Y}}
.$$
If $f \in F$, $x, x{'} \in X$, and $y, y{'} \in Y$, and if $z := \big((x, y), (x{'}, y{'})\big)$, then
\begin{align*}
\big[\big(f \circ \pi^{\mathsmaller{X}}\big) \circ +^{\mathsmaller{X \times Y}}\big]\big(z\big) & :=
\big(f \circ \pi^{\mathsmaller{X}}\big)\big(x +^{\mathsmaller{X}} x{'}, y +^{\mathsmaller{Y}} y{'}\big)\\
& := f(x +^{\mathsmaller{X}} x{'})\\
& := (f \circ +^{\mathsmaller{X}})(x, x{'})\\
& := (f \circ +^{\mathsmaller{X}})\big(\pi^{\mathsmaller{X}}(x, y), \pi^{\mathsmaller{X}}(x{'}, y{'})\big)\\
& := (f \circ +^{\mathsmaller{X}}) \bigg(\big[\pi^{\mathsmaller{X}} \circ \pi_1^{\mathsmaller{X \times Y}}\big]
\big(z\big), \big[\pi^{\mathsmaller{X}} \circ \pi_2^{\mathsmaller{X \times Y}}\big]
\big(z\big)\bigg)\\
& := \big[(f \circ +^{\mathsmaller{X}}) \circ h \big)\big](z),
\end{align*}
where by Corollary~\ref{crl: corprod3} the function $h : [(X \times Y) \times (X \times Y)] \to X \times X$, where
$h := \big[\pi^{\mathsmaller{X}} \circ \pi_1^{\mathsmaller{X \times Y}}\big]
\times \big[\pi^{\mathsmaller{X}} \circ \pi_2^{\mathsmaller{X \times Y}}\big]$
is a Bishop morphism. Since $\big(f \circ \pi^{\mathsmaller{X}}\big) \circ +^{\mathsmaller{X \times Y}} = 
(f \circ +^{\mathsmaller{X}}) \circ h$
\begin{center}
\begin{tikzpicture}

\node (E) at (0,0) {$(X \times Y) \times (X \times Y)$};
\node[right=of E] (F) {$X \times X$};
\node[right=of F] (A) {$X$};
\node[right=of A] (B) {$\Real$,};

\draw[->] (E)--(F) node [midway,above] {$h$};
\draw[->] (F)--(A) node [midway,above] {$+^{\mathsmaller{X}}$};
\draw[->] (A)--(B) node [midway,above] {$f$};
%\draw[->,bend right] (E) to node [midway,below] {$+_{x_0}^1$} (A) ;
\draw[->,bend right] (E) to node [midway,below] {$\mathsmaller{\big(f \circ \pi^{\mathsmaller{X}}\big) 
\circ +^{\mathsmaller{X \times Y}}}$} (B) ;

\end{tikzpicture}
\end{center}
we get $\big(f \circ \pi^{\mathsmaller{X}}\big) \circ +^{\mathsmaller{X \times Y}} \in 
\Mor\big([\C F \times \C G] \times [\C F \times \C G], \C R\big)$ as a composition of Bishop morphisms, hence
$\big(f \circ \pi^{\mathsmaller{X}}\big) \circ +^{\mathsmaller{X \times Y}} \in [F \times G] \times [F \times G]$.
Working similarly, we get $\big(g \circ \pi^{\mathsmaller{Y}}\big) \circ +^{\mathsmaller{X \times Y}} 
\in [F \times G] \times [F \times G]$. By the $\V$-lifting of morphisms we also have that
$$-^{\mathsmaller{X \times Y}} \in \Mor\big(\C F \times \C G, \C F \times \C G\big) \TOT
\forall_{f \in F}\forall_{g \in G}\bigg(\big(f \circ \pi^{\mathsmaller{X}}\big) \circ -^{\mathsmaller{X \times Y}} \in
F \times G \ \& \ \big(g \circ \pi^{\mathsmaller{Y}}\big) \circ -^{\mathsmaller{X \times Y}} 
\in F \times G\bigg).$$
If $f \in F$, $x \in X$, and $y \in Y$, then
$$\big[(\big(f \circ \pi^{\mathsmaller{X}}\big) \circ -^{\mathsmaller{X \times Y}}\big](x, y) :=
(\big(f \circ \pi^{\mathsmaller{X}}\big)\big(-^{\mathsmaller{X}} x, -^{\mathsmaller{Y}}y\big) := f(-^{\mathsmaller{X}} x)
:= f_-(x) := \big(f_- \circ \pi^{\mathsmaller{X}}\big)(x, y)$$ 
i.e., $(\big(f \circ \pi^{\mathsmaller{X}}\big) \circ -^{\mathsmaller{X \times Y}} = f_- \circ \pi^{\mathsmaller{X}} \in
F \times G$, since $f_- \in F$. Similarly, $\big(g \circ \pi^{\mathsmaller{Y}}\big) \circ -^{\mathsmaller{X \times Y}} 
\in F \times G$. 
\end{proof}

Since the projections $\pi^{\mathsmaller{X}}, \pi^{\mathsmaller{Y}}$ are homomorphisms, they are Bishop homomorphisms.
By the universal property of the product Bishop topology, $\B {\C F \times \C G}$ is the product in 
%the category
$\BTopGrp$.

\section{Closed subsets in Bishop topological groups}
\label{sec: closed}

\begin{proposition}\label{prp: closed1}
Let $C \subseteq X$ and $x_0 \in X$.\\[1mm] 
\normalfont (i)
\itshape If $C$ is closed, then $-C := \{-c \mid c \in C\}$ is closed.\\[1mm]
\normalfont (ii)
\itshape $\overline{-C} = - \overline{C}$.\\[1mm]
\normalfont (iii)
\itshape If $C$ is closed, then $x_0 + C := \{x_0 + c \mid c \in C\}$ is closed.\\[1mm]
\normalfont (iv)
\itshape $\overline{x_0 + C} = x_0 + \overline{C}$.
\end{proposition}

\begin{proof}
 (i) We suppose that $u \in \overline{-C}$ i.e., if $f(u) > 0$, there is $w \in -C$ such that $f(w) > 0$, for every 
 $f \in F$, and we show that $u \in -C$ i.e., $-u \in C$. Since $C$ is closed, it suffices to show that 
 $-u \in \overline{C}$. Let $f \in F$ such that $f(-u) > 0 :\TOT f_{-}(u) > 0$. By our hypothesis on $u$ there is 
 $w \in -C$ such that $f_{-}(w) := f(-w) > 0$, and $-w \in C$.\\
 (ii) Since $C \subseteq \overline{C}$, we get $-C \subseteq -\overline{C}$. Since $\overline{C}$ is closed, by (i) 
 $-\overline{C}$ is also closed, hence $\overline{-C} \subseteq \overline{-\overline{C}} = -\overline{C}$. To show the 
 converse inclusion $-\overline{C} \subseteq \overline{-C}$, let $x \in -\overline{C}$, hence $-x \in \overline{C}$
 i.e., if $f(-x) > 0$, there is $u \in C$ with $f(c) > 0$, for every $f \in F$. We show that $x \in \overline{-C}$. 
 Let $f \in F$ with $f(x) = f(-(-x)) = f_{-}(-x) > 0$. By our hypothesis on $-x$, there $u \in C$ with $f_{-}(c) :=
 f(-u) > 0$, and $-u \in -C$.\\
 (iii) We suppose that $y \in \overline{x_0 + C}$ i.e., if $f(y) > 0$, there is $c \in C$ such that $f(x_0 + c) > 0$,
 for every  $f \in F$, and we show that $y \in x_0 + C$ by showing that $-x_0 + y \in C$. As $C$ is closed,
 it suffices to show that  $-x_0 + y \in \overline{C}$. Let $f \in F$ such that $f(-x_0 + y) > 0 \TOT f^1_{-x_0}(y) > 0$.
 We show that there is $c \in C$ such that $f(c) > 0$. By our hypothesis on $y$, there is $c \in C$ such that
 $f^1_{-x_0}(x_0 + c) := f(-x_0 + x_0 + c) = f(c) > 0$.\\
 (iv) Since $\overline{C}$ is closed, and $x_0 + C \subseteq x_0 + \overline{C}$, we get 
 $\overline{x_0 + C} \subseteq \overline{x_0 + \overline{C}} = x_0 + \overline{C}$. For the converse inclusion, let
 $x := x_0 + y$ with $y \in \overline{C}$. We show that $x \in \overline{x_0 + C}$. Let $f \in F$ such that 
 $f(x_0 + y) > 0 \TOT f^1_{x_0}(y) > 0$. We find $c \in C$ such that $f(x_0 + c) > 0$. 
 By our hypothesis on $y$ though, there is $c \in C$ such that $f^1_{x_0}(c) := f(x_0 + c) > 0$.
\end{proof}

\begin{corollary}\label{cor: corclosed1}
$F$ separates the points of $X$ if and only if $\{0^X\}$ is closed.
\end{corollary}

\begin{proof}
If $\{0^X\}$ is closed, then by Proposition~\ref{prp: closed1}(iii) $\{x\} = x + \{0^X\}$ is closed, for every $x \in X$.
By Proposition~\ref{prp: tight} $F$ is separating. The converse follows immediately from  Proposition~\ref{prp: tight}.
\end{proof}

\begin{proposition}\label{prp: closed2}
 If $C$ is an open subgroup of $X$, then $C$ is closed in $X$
\end{proposition}

\begin{proof}
 Let $x \in \overline{C}$ i.e., if $f(x) > 0$, there is $u \in C$ such that $f(u) > 0$, 
 for every $h \in F$. We show that $x \in C$. Since $C$ is a subgroup of $X$, we have that $0 \in C$. Since $C$ is 
 open in $X$, there is $g \in F$ such that $g(0) > 0$ and $U(g) \subseteq C$. Since $g_{-x}^1 \in F$ and
 $$g_{-x}^1(x) := g(-x + x) = g(0) > 0,$$
 by our hypothesis on $x$ there is $u \in C$ such that $g_{-x}^1(u) := g(-x + u) > 0$. Since $U(g) \subseteq C$,
 we get $-x + u \in C$, and since $C$ is a subgroup of $X$, we get $x \in C$.
 \end{proof}

Classically, $C$ is closed, since its complement in $X$ is the open set $\bigcup \{x + C \mid x \notin C\}$,
where $x + C$ is open, for every $x \in X$, as $C$ is open (this holds also constructively). The double use of negation
in the classical proof is replaced here by the clear algorithm of the previous proof.

\begin{lemma}\label{lem: corbgroup6}
%Let $\B {\C F} := (X, +, 0, - ; F)$ be a Bishop topological group. 
The \textit{commutator map} $\com : X \times X
\to X$ is defined by $\com(x, y) := x + y - x - y,$ , for every $(x, y) \in X \times X$.\\[1mm]
\normalfont (i)
\itshape $\com \in  \Mor(\C F \times \C F, \C F)$.\\[1mm]
\normalfont (ii)
\itshape If $x, y \in X$, then $x + y =_X y + x \TOT \com(x, y) = 0^\mathsmaller{X}$.\\[1mm]
\normalfont (iii)
\itshape If $x \in X$, the mapping $\com_x : X \to X$, where $\com_x(y) := \com(x, y)$, for every $y \in X$, is in
$\Mor(\C F, \C F)$, and for every $f \in F$ the composition $f \circ \com_x \in F$
\begin{center}
\begin{tikzpicture}

\node (E) at (0,0) {$X$};
\node[right=of E] (F) {$X$};
%\node[right=of F] (A) {$X$};
\node[right=of F] (B) {$\Real$.};

\draw[->] (E)--(F) node [midway,above] {$\mathsmaller{\com_x}$};
%\draw[->] (F)--(A) node [midway,above] {$f$};
\draw[->] (F)--(B) node [midway,above] {$f$};
%\draw[->,bend right] (E) to node [midway,below] {$+_{x_0}^1$} (A) ;
\draw[->,bend right] (E) to node [midway,below] {$\mathsmaller{f \circ \com_x}$} (B) ;

\end{tikzpicture}
\end{center}
\normalfont (iv)
\itshape If $x, y \in X$, then $\com_x(y) = - \com_y(x)$.\\[1mm]
\normalfont (v)
\itshape If $x, y \in X$, then $\com_x(y) = 0^\mathsmaller{X} \TOT \com_y(x) = 0^\mathsmaller{X}$.
\end{lemma}

\begin{proof}
(i) By the definition of the product Bishop topology $\pi_1, \pi_2 \in \Mor(\C F \times \C F,
\C F)$. Since $\com := \pi_1 + \pi_2 - \pi_1 - \pi_2$, by Proposition~\ref{prp: bgroup6} we get 
$\com \in \Mor(\C F \times \C F, \C F)$.\\
(ii) and (iii) The proof for (ii) is immediate. By (i) and Proposition~\ref{prp: prod6}(iii) 
 $\com_x \in \Mor(\C F, \C F)$, hence by the definition of a Bishop morphism $f \circ \com_x \in F$, for every 
 $f \in F$.\\
(iv) is trivial and (v) follows immediately from (iv).
%$\com_x(y) = 0^\mathsmaller{X} \TOT x + y = y + x \TOT $, hence $\com_y(x) := y + x - x - y = 0^\mathsmaller{X}$. 
%The converse implication is shown similarly.
\end{proof}

\begin{lemma}\label{lem: 2corbgroup6}
%Let $\B {\C F} := (X, +, 0, - ; F)$ be a Bishop topological group, 
Let $x \in X$ and $H \subseteq X$. The maps
$\normal_x : X \to X$ and $\Normal_x : X \to X$ are defined, for every $x \in X$, respectively, by 
$\normal_x(y) := x + y - x$ and $\Normal_x(y) := y + x - y$.
Let $\normal_x^{\mathsmaller{H}}, \Normal_x^{\mathsmaller{H}} : H \to X$ be the restrictions
of $\normal_x$ and $\Normal_X$ to $H$, respectively.\\[1mm]
\normalfont (i)
\itshape If $x, y \in X$, then $\normal_x(y) = \Normal_y(x)$.\\[1mm] 
\normalfont (ii)
\itshape $\normal_x \in  \Mor(\C F, \C F)$ and $\Normal_x \in  \Mor(\C F, \C F)$.\\[1mm]
\normalfont (iii)
\itshape If $H \leq X$, then $H$ is normal if and only if $\normal_x^{\mathsmaller{H}}
: H \to H$, for every $x \in X$.\\[1mm]
\normalfont (iv)
\itshape  If $f \in F$, the compositions $f \circ \normal_x \in F$ and $f \circ \Normal_x \in F$.
\begin{center}
\begin{tikzpicture}

\node (E) at (0,0) {$X$};
\node[right=of E] (F) {$X$};
%\node[right=of F] (A) {$X$};
\node[right=of F] (B) {$\Real$.};
\node[left=of E] (G) {$X$};
\node[left=of G] (A) {$\Real$};

\draw[->] (E)--(F) node [midway,above] {$\mathsmaller{\normal_x}$};
\draw[->] (G)--(A) node [midway,above] {$\ \ \mathsmaller{f}$};
\draw[->] (F)--(B) node [midway,above] {$\mathsmaller{f}$};
%\draw[->,bend right] (E) to node [midway,below] {$+_{x_0}^1$} (A) ;
\draw[->,bend right] (E) to node [midway,below] {$\mathsmaller{f \circ \normal_x}$} (B) ;
\draw[->] (E)--(G) node [midway,below] {$ \ \ \mathsmaller{\Normal_x}$};
\draw[->,bend right] (E) to node [midway,above] {$\mathsmaller{f \circ \Normal_x}$} (A) ;

\end{tikzpicture}
\end{center}
\normalfont (v)
\itshape  If $H$ is normal, then $\normal_x^{\mathsmaller{H}} \in \Mor\big(\C F_{|H}, \C F_{|H}\big)$.\\[1mm]
\normalfont (vi)
\itshape  If $\Normal_x^{\mathsmaller{H}} : H \to H$, then $\Normal_x^{\mathsmaller{H}} \in \Mor\big(\C F_{|H},
\C F_{|H}\big)$.
\end{lemma}

\begin{proof}
(i) The proof is immediate. For the proof of (ii), the function $c_x : X \to X$, 
defined by $c_x(y) := x$, for every $y \in X$, is in 
$\Mor(\C F, \C F)$; if $f \in F$, then $(f \circ c_x)(y) := f(x)$, for every $y \in X$, hence $f \circ c_x = 
\overline{f(x)}^{\mathsmaller{X}} \in F$. The identity map $\id_X$ on $X$ is also in $\Mor(\C F, \C F)$.
Since $\normal_x := c_x + \id_X - c_x$, by Proposition~\ref{prp: bgroup6} $\normal_x \in  \Mor(\C F, \C F)$.
Since $\Normal_x := \id_X + c_x - \id_X$, we get $\Normal_x \in  \Mor(\C F, \C F)$.
(iii) and (iv) are immediate to show.  For the proof of (v), by the $\V$-lifting of morphisms we have that
$\normal_x^{\mathsmaller{H}} \in \Mor\big(\C F_{|H}, \C F_{|H}\big)$ if and only if 
$\forall_{f \in F}\big(f_{|H} \circ \normal_x^{\mathsmaller{H}} \in F_{|H}\big).$
If $f \in F$ and $v \in H$, then
$\big(f_{|H} \circ \normal_x^{\mathsmaller{H}}\big)(v) := f_{|H}(x + v -x)
:= f(x + v -x) := \big(f \circ \normal_x\big)(v) := \big(f \circ \normal_x\big)_{|H}(v)$
i.e., $f_{|H} \circ \normal_x^{\mathsmaller{H}} = \big(f \circ \normal_x\big)_{|H} \in F_{|H}$, since  
$f \circ \normal_x \in F$. For the proof of (vi), we proceed as in the proof of (v).
\end{proof}

\begin{theorem}\label{thm: bnbh1}
%Let $\B {\C F} := (X, +, 0, - ; F)$ be a Bishop topological group and 
Let $H$ be a subgroup of $X$.\\[1mm]
\normalfont (i)
\itshape The closure $\overline{H}$ of $H$ is also a subgroup of $X$.\\[1mm]
\normalfont (ii)
\itshape If $F$ is a separating Bishop topology and $H$ is abelian, then $\overline{H}$ is abelian.\\[1mm]
\normalfont (iii)
\itshape If $H$ is normal, then $\overline{H}$ is normal.
\end{theorem}

\begin{proof}
(i) Since $H \subseteq \overline{H}$ and $0 \in H$, we get $0 \in \overline{H}$.
Let $x, y \in \overline{H}$, where by definition
$$x \in \overline{H} :\TOT \forall_{f \in F}\big(f(x) > 0 \To \exists_{v \in H}(f(v) > 0)\big),$$
$$y \in \overline{H} :\TOT \forall_{f \in F}\big(f(y) > 0 \To \exists_{v \in H}(f(v) > 0)\big).$$
We show that $x + y \in \overline{H}$ i.e.,
$\forall_{f \in F}\big(f(x+y) > 0 \To \exists_{v \in H}(f(v) > 0)\big)$.
Let $f \in F$ such that $f(x+y) > 0$. By Proposition~\ref{prp: bgroup3}(iii) the map $f_y^2 : X \to \Real \in F$, where
$f_y^2(u) := f(u + y)$, for every $u \in X$. By hypothesis, $f_y^2(x) > 0$. Since $x \in \overline{H}$, there is
$z \in H$ such that 
$$f_y^2(z) > 0 :\TOT f(z + y) > 0 :\TOT f_z^1(y) > 0,$$
where by Proposition~\ref{prp: bgroup3}(iii) the map $f_z^1 : X \to \Real \in F$, where
$f_z^1(u) := f(z + u)$, for every $u \in X$, is in $F$.
Since $y \in \overline{H}$, there is $w \in H$ such that 
$f_z^1(w) > 0 :\TOT f(z + w) > 0$.
Since $H \leq X$, we get $z + w \in H$, which is what we need to show. Next we show 
that $-x \in \overline{H}$ i.e.,
$$\forall_{f \in F}\big(f(-x) > 0 \To \exists_{v \in H}(f(v) > 0)\big).$$
Let $f \in F$ such that $f(-x) > 0 \TOT f_-(x) > 0$. Since $f_- \in F$ and $x \in \overline{H}$, there is $v \in H$ 
such that 
$f_-(v) := f(-v) > 0$. Since $H \leq X$, we get $-v \in H$, and our proof is completed.\\[1mm]
(ii) \textbf{Case I}: $x \in H$ and $y \in \overline{H}$.\\[1mm]
We show that
$\com_x(y) =_{\mathsmaller{X}} 0^{\mathsmaller{X}}$.
Suppose that $\com_x(y) \neq_{\mathsmaller{F}} 0^{\mathsmaller{X}}$. By Proposition 5.2.5 in~\cite{Pe15} there is
$f \in F$ such that $f\big(\com_x(y)\big) = 1$ and $f\big(0^{\mathsmaller{X}}\big) = 0$. Since $y \in \overline{H}$, and
$(f \circ \com_x)(y) = 1 > 0$, and $f \circ \com_x \in F$, there is $v \in H$ such that $(f \circ \com_x)(v) > 0$. Since
$x, v \in H$ and $H$ is abelian, we have that $\com_x(v) = 0^{\mathsmaller{X}}$, hence $0 = f\big(0^{\mathsmaller{X}}\big) =
(f \circ \com_x)(v) > 0$, which is a contradiction. Since $F$ is separating, the canonical apartness relation
$\neq_F$ of $F$ is tight (see Proposition 5.1.3 in~\cite{Pe15}), hence the negation of 
$\com_x(y) \neq_{\mathsmaller{F}} 0^{\mathsmaller{X}}$ 
implies that $\com_x(y) =_{\mathsmaller{X}} 0^{\mathsmaller{X}}$.\\[1mm]
\textbf{Case II}: $x \in \overline{H}$ and $y \in \overline{H}$.\\[1mm]
We show that $\com_x(y) =_{\mathsmaller{X}} 0^{\mathsmaller{X}}$.
Suppose that $\com_x(y) \neq_{\mathsmaller{F}} 0^{\mathsmaller{X}}$. As in the previous case,
there is
$f \in F$ such that $f\big(\com_x(y)\big) = 1$ and $f\big(0^{\mathsmaller{X}}\big) = 0$. Since $y \in \overline{H}$, and
$(f \circ \com_x)(y) = 1 > 0$, there is $v \in H$ such that $(f \circ \com_x)(v) > 0$. By case I we have that 
$\com_v(x) = 0^{\mathsmaller{X}}$, hence by Lemma~\ref{lem: corbgroup6}(v) we get $\com_x(v) = 0^{\mathsmaller{X}}$,
hence $0 = f\big(0^{\mathsmaller{X}}\big) = (f \circ \com_x)(v) > 0$, which is a contradiction. Since $F$ is separating,
we conclude, similarly to Case I, that $\com_x(y) =_{\mathsmaller{X}} 0^{\mathsmaller{X}}$.\\[1mm]
(iii) By Lemma~\ref{lem: 2corbgroup6}(ii) it suffices to show that $\normal_x^{\mathsmaller{\overline{H}}} : \overline{H}
\to \overline{H}$, for every $x \in X$. If $x \in X$ and $y \in \overline{H}$, we show that 
$$\normal_x^{\mathsmaller{\overline{H}}}(y) := x + y - x \in \overline{H}.$$
Let $f \in F$ such that 
$f(x + y - x) > 0 \TOT (f \circ \normal_x)(y) > 0$.
We need to find $u \in H$ such that $f(u) > 0$.
Since $y \in \overline{H}$ and $f \circ \normal_x \in F$, there is $v \in H$ with
$$(f \circ \normal_x)(v) := f(x + v - x) > 0.$$
Since $H$ is normal and $v \in H$, we have that $u := x + v - x \in H$ and $f(u) > 0$.
\end{proof}

As we explain in~\cite{Pe15}, section 5.8, by the Stone-\v Cech theorem for Bishop spaces we have that
the separating hypothesis on $F$ in case (ii) of the previous proposition is not a serious restriction on $F$.
Note that classically, if $X$ is a Hausdorff topological group, then if $H$ is an abelian subgroup, then $\overline{H}$ is
abelian. The $F$-version of a Hausdorff topology is the following (see~\cite{Pe15}, section 5.2):
if $\neq$ is a given apartness relation on $X$, then $F$ is $\neq$-Hausdorff, if $\neq \subseteq \neq_F$. In Proposition
5.2.3. of~\cite{Pe15} we show that this is equivalent to the positive $\neq$-version of the induced neighborhood structure
being Hausdorff. If $F$ is $\neq$-Hausdorff, $n$-many pairwise $\neq$-apart points of $X$ are mapped to
given $n$-many real numbers. This is essential to the proof we gave above. So, we
capture computationally the requirement of a Hausdorff topology for $\overline{H}$
being abelian. 
%Cases (ii) and (ii) above are written as one as follows: if $F$ is a Bishop topology on $X$ such that 
%$\neq_F$ is tight, and $H$ is abelian, then $\overline{H}$ is abelian. 
% A Bishop topology $F$ on $X$ is pointed, if for every $x \in X$ there is $f \in F$ with $\zeta(f) = \{x\}$.
% If $F$ is a pointed Bishop topology, then its canonical apartness relation $\neq_F$ is tight 
% (Proposition 5.1.5. in~\cite{Pe15}), hence $F$ is separating.
% %and if arguing exactly as in the proof of (ii) we get that 
% Hence, if $F$ is a pointed Bishop topology and $H$ is abelian, then $\overline{H}$ is abelian.
By Proposition~\ref{prp: tight} the tightness of $\neq_F$ in a Bishop topological group
implies that the induced topology is Hausdorff in the classical sense (classically, a topological group is Hausdorff 
if and only if there is a closed singleton). 
In the next proposition the subset $H$ of $X$ is extensional i.e.,
$H$ is closed under the given equality $=_X$ on $X$, so that the defining property of $\Normal_X(H)$
in the use of the separation scheme is also extensional.

\begin{theorem}\label{thm: bnbh2}
%Let $\B {\C F} := (X, +, 0, - ; F)$ be a Bishop topological group and
Let $H$ be an extensional subset of $X$. 
The normalizer $\Normal_X(H)$ and the center $\Center_X(H)$ of $H$ in $X$, are defined by
%defined, respectively, by
$$\Normal_X(H) := \big\{x \in X \mid \Normal_x^{\mathsmaller{H}} : H \to H\big\},$$
$$\Center_X(H) := \big\{x \in X \mid \forall_{v \in H}\big(\com_x(v) = 0^{\mathsmaller{X}}\big)\big\}.$$
\normalfont (i)
\itshape If $H$ is closed, then $\Normal_X(H)$ is closed.\\[1mm]
\normalfont (ii)
\itshape If $F$ is separating, then $\Center_X(H)$ is closed.
\end{theorem}

\begin{proof}
(i) We suppose that $H$ is closed i.e.,
$$(\Hyp_1) \ \ \ \ \ \ \ \ \ \ \ \ \ \ \ \ \ \ \ \ \ \ \ \ \ \ \ \ \  \forall_{x \in X}\bigg(\forall_{f \in F}\big(f(x) > 0 \To 
\exists_{v \in H}\big(f(v) > 0\big)\big) \To x \in H\bigg), \ \ \ \ \ \ \ \ \ \ \ \ \ \ \ \ \ \ \ \ \ \ \ \ \ \  $$
and we show that $\Normal_X(H)$ is closed i.e.,
$$(\Goal_1) \ \ \ \ \ \ \ \ \ \ \ \ \ \ \ \ \  \forall_{x \in X}\bigg(\forall_{f \in F}\big(f(x) > 0 \To 
\exists_{u \in \Normal_X(H)}\big(f(u) > 0\big)\big) \To x \in \Normal_X(H)\bigg). \ \ \ \ \  \ \ \ \ \ \ \ \ \ \ \ \ \ 
\  $$
For that we fix some $x \in X$ and we suppose that 
$$(\Hyp_2) \ \ \ \ \ \ \ \ \ \ \ \ \ \ \ \ \ \ \ \ \ \ \ \ \ \ \ \ \ \ \ \ \  \forall_{f \in F}\big(f(x) > 0 \To 
\exists_{u \in \Normal_X(H)}\big(f(u) > 0\big)\big), \ \ \ \ \ \ \ \ \ \ \ \ \ \ \ \ \ \ \ \ \ 
\ \ \ \ \ \ \ \ \ \ \ \ \ \ \ \ \ \ \ \ \ \ \ \ $$
and we show that
$(\Goal_2) \ \ x \in \Normal_X(H) :\TOT \Normal_x^{\mathsmaller{H}} : H \to H$. Let $v \in H$ be fixed, and we show 
$(\Goal_3) \ \ \Normal_x^{\mathsmaller{H}}(v) := v + x - v \in H.$ By $\Hyp_1$ it suffices to show the following
$$(\Goal_4) \ \ \ \ \ \ \ \ \ \ \ \ \ \ \ \ \ \ \ \ \ \ \ \ \ \ \ \ \ \ \ \ \
\forall_{f \in F}\bigg(f\big(\Normal_x^{\mathsmaller{H}}(v)\big) > 0 \To 
\exists_{w \in H}\big(f(w) > 0\big)\bigg). \ \ \ \ \ \ \ \ \ \ \ \ \ \ \ \ \ \ \ \ \ 
\ \ \ \ \ \ \ \ \ \ \ \ \ \ \ \ \ \ \ \ \ \ \ \ $$
If we fix $f \in F$, we suppose that 
$$(\Hyp_3) \ \ \ \ \ \ \ \ \ \ \ \ \ \ \ \   f\big(\Normal_x^{\mathsmaller{H}}(v)\big) > 0  \TOT f\big(\normal_v(x)\big) > 0
\TOT (f \circ \normal_v))(x) > 0, \ \ \ \ \ \ \ \ \ \ \ \ \ \ \ \ \ \ \ \ \ \ \ \ \  $$
and we show $(\Goal_5) \ \ \exists_{w \in H}\big(f(w) > 0\big)$. Since $f \circ \normal_v \in F$, by $(\Hyp_2)$ there is 
$u \in \Normal_X(H)$ with $(f \circ \normal_v)(u) := f(v +u - v) > 0$. Since $u \in \Normal_X(H)$, 
$\Normal_u^{\mathsmaller{H}} : H \to H$, and since $v \in H$, we get $\Normal_u^{\mathsmaller{H}}(v) := v + u - v \in H$.
Hence, $w := v + u - v \in H$ and $f(w) > 0$.\\
(ii) We fix $x \in X$, we suppose that
$$(\Hyp_1^*) \ \ \ \ \ \ \ \ \ \ \ \ \ \ \ \ \ \ \ \ \ \ \ \ \ \ \ \ \ \ \ \ \  \forall_{f \in F}\big(f(x) > 0 \To 
\exists_{u \in \Center_X(H)}\big(f(u) > 0\big)\big), \ \ \ \ \ \ \ \ \ \ \ \ \ \ \ \ \ \ \ \ \ 
\ \ \ \ \ \ \ \ \ \ \ \ \ \ \ \ \ \ \ \ \ \ \ \ $$
and we show $(\Goal_1^*) \ \ x \in \Center_X(H) :\TOT \forall_{v \in H}\big(\com_x(v) = 0^{\mathsmaller{X}}\big)$. 
Let $v \in H$ be fixed. Since $F$ is a separating Bishop topology, it suffices to 
prove $\neg \big(\com_x(v) \neq_F 0^{\mathsmaller{X}}\big)$. If we suppose that
$\com_x(v) \neq_F 0^{\mathsmaller{X}}$, there is $f \in F$ such that 
$$f\big(\com_x(v)\big) = 1 > 0 \ \ \ \& \ \ \ f\big(0^{\mathsmaller{X}}\big) = 0.$$
By Lemma~\ref{lem: corbgroup6}(iv) we get
$$1 = f\big(\com_x(v)\big) = f\big(-\com_v(x)\big) := f_-\big(\com_v(x)\big) \ \ \ \& \ \ \ 0 = 
f\big(0^{\mathsmaller{X}}\big) = f_-\big(0^{\mathsmaller{X}}\big).$$
Since $f_ \in F$, we have that $f_- \circ \com_v \in F$ and $(f_- \circ \com_v)(x) > 0$. By $(\Hyp_1^*)$ there is 
$u \in \Center_X(H)$ such that $(f_- \circ \com_v)(u) > 0$. Since $u \in \Center_X(H)$, we get
$\forall_{w \in H}\big(\com_u(w) = 0^{\mathsmaller{X}}\big)$. Moreover, 
$(f_- \circ \com_v)(u) := f_-(v + u - v - u) > 0$. 
Since $v \in H$, we get
$$v + u - v - u := \com_v(u) = - \com_u(v) = - 0^{\mathsmaller{X}} = 0^{\mathsmaller{X}},$$
hence $f_-\big(0^{\mathsmaller{X}}\big) > 0$, which contradicts the previously established equality 
$f_-\big(0^{\mathsmaller{X}}\big) = 0^{\mathsmaller{X}}$.
\end{proof}

\begin{proposition}\label{prp: kernel}
% Let $\B {\C F} := (X, +^{\mathsmaller{X}}, 0^{\mathsmaller{X}}, -^{\mathsmaller{X}} ; F)$ and
% $\B {\C G} := (Y, +^{\mathsmaller{Y}}, 0^{\mathsmaller{Y}}, -^{\mathsmaller{Y}} ; G)$ be Bishop topological groups.
If $G$ is a separating Bishop topology on $Y$ and $h \in \MOR(\B {\C F}, \B {\C G})$, the kernel 
$\Ker(h) := \big\{x \in X \mid h(x) =_Y 0^{\mathsmaller{Y}}\big\}$ of $h$ is a closed set in $\C F$.
\end{proposition}

\begin{proof}
Let $x \in X$ such that 
$\forall_{f \in F}\big(f(x) > 0 \To \exists_{v \in \Ker(h)}\big(f(v) > 0\big)\big)$.
Since $\neq_G$ is tight, it suffices to show that $\neg \big(h(x) \neq_G 0^{\mathsmaller{Y}}\big)$. If 
$h(x) \neq_G 0^{\mathsmaller{Y}}$, there is $g \in G$ such that $g(h(x)) = 1 > 0$ and $g\big(0^{\mathsmaller{Y}}\big) = 0$.
Since $h \in \Mor(\C F, \C G)$, we have that $g \circ h \in F$. As $(g \circ h)(x) > 0$, 
there is $v \in \Ker(h)$ such that $0 = g\big(0^{\mathsmaller{Y}}\big) = g(h(v)) := (g \circ h)(v) > 0$, which 
is a contradiction.
\end{proof}

\begin{theorem}[Characterisation of a closed (open) subgroup]\label{thm: chars}
Let $C$ be a subgroup of $X$.\\[1mm]
\normalfont (i)
\itshape $C$ is closed if and only if there is an open set $O$ in $X$ such that $O \cap C$ is inhabited 
and closed in $O$.\\[1mm]
\normalfont (ii)
\itshape 
$C$ is open if and only if there is an inhabited, open set $O$ in $X$ such that $O \subseteq C$.
\end{theorem}

\begin{proof}
(i) Let $O$ be open in $X$ such that $O \cap C$ is inhabited 
and closed in $O$. We show that $C$ is closed. Suppose that $x \in \overline{C}$ i.e., i.e., if $f(x) > 0$, 
there is $u \in C$ such that $f(u) > 0$, for every $h \in F$. We prove that $x \in C$. 
Let $c_0 \in O \cap C$. Since $O$ is open, there is $g \in F$
such that $g(c_0) > 0$ and $U(g) \subseteq O$. Since $g^1_{-x + c_0} \in F$ and
$$g^1_{-x + c_0}(x) := g(x-x+c_0) = g(c_0) > 0,$$
by our hypothesis on $x$ there is $c \in C$ such that 
$$g^1_{-x + c_0}(c) := g(c-x+c_0) > 0.$$
As $U(g) \subseteq O$, we get $c-x+c_0 \in O$. The hypothesis ``$O \cap C$ is closed in $O$'' means 
$$\forall_{z \in O}\big(z \in \overline{O \cap C} \To z \in O \cap C\big).$$
Let $z_0 := c-x+c_0 \in O$. We show that $z_0 \in \overline{O \cap C}$, hence $z_0 \in O \cap C$. Since $C$ is a subspace,
and since then $z_0 \in C$, we get the required membership $x \in C$. 
To show that $z_0 \in \overline{O \cap C}$, let $f \in F$ such that $f(z_0) > 0$.
We find $w \in O \cap C$ such that $f(w) > 0$. Since $f(z_0) > 0$ and $g(z_0) > 0$, we have that $(f \wedge g)(z_0) > 0$
(see~\cite{BV11}, p.~57). By Theorem~\ref{thm: bnbh1}(i) $\overline{C}$ is a subgroup of $X$. Since
$C \subseteq \overline{C}$, and $c, c_0, -x \in \overline{C}$, we get $z_0 \in \overline{C}$. Since $f \wedge g \in F$ and
$(f \wedge g)(z_0) > 0$, there is $w \in C$ such that 
$$(f \wedge g)(w) > 0.$$
Since $g(w) \geq (f \wedge g)(w) > 0$, we get $w \in O$, hence $w \in O \cap C$. 
Since $f(w) \geq (f \wedge g)(w) > 0$, we conclude that $f(w) > 0$, as required. For the converse,
if $C$ is closed, then $X$ is open, $C = C \cap X$ is inhabited by $0$ and it is trivially closed in $X$. \\
(ii) Let $O$ be open in $X$ such that $O \subseteq C$, and let $c_0 \in O$. Suppose that $g \in F$ with
$g(c_0) > 0$ and $U(g) \subseteq O$. Since also $U(g) \subseteq C$, we get $c_0 \in C$. Let $c \in C$. The function
$g^1_{-c + c_0} \in F$ and $g^1_{-c + c_0}(c) = g(c_0) > 0$. We show that $U(g^1_{-c + c_0}) \subseteq C$, hence, since 
$c$ is an arbitrary element of $C$, we conclude that $C$ is open. Let $u \in X$ such that 
$g^1_{-c + c_0}(u) := g(u-c+c_0) > 0$. As $U(g) \subseteq O \subseteq C$, we get $u-c+c_0 \in C$. As $c, c_0 \in C$ and $C$
is a subgroup of $X$, we have that $u \in C$. For the converse, if $C$ is open, then $C$ is an inhabited open set included 
in $C$.
\end{proof}

The classical proof of Theorem~\ref{thm: chars}(i) is based on multiple use of negation (see~\cite{Kr17}). As usual in 
constructive mathematics, we replaced the ``non-empty intersection'' of $O$ and $C$ in 
Theorem~\ref{thm: chars}(i) with $O \between C$, and 
the ``non-emptyness'' of $O$ in Theorem~\ref{thm: chars}(ii) with the stronger inhabitedness of $O$. Although, in general, 
it is not possible to show that the $F$-complement $X \setminus_F C$ of a closed set in a Bishop space $X$ is equal to 
$X \setminus C$, there is a number of cases in the theory of Bishop topological 
groups where this is possible.

\begin{corollary}\label{cor: corchar1}
If $C$ is closed in $X$, such that $X \setminus C$ is a subgroup of $X$ and $X \setminus_F C$ is inhabited, then
$X \setminus C = X \setminus_F C$ and $X \setminus C$ is clopen.
\end{corollary}

\begin{proof}
By Theorem~\ref{thm: cr} we have that $X \setminus_F C$ is open in $X$ and $X \setminus_F C \subseteq X \setminus C$.
Since $X \setminus_F C$ is inhabited and $X \setminus C$ is a subgroup of $X$, by Theorem~\ref{thm: chars}(ii) we have that
$X \setminus C$ is open. As $X \setminus C \subseteq X \setminus C$, by Theorem~\ref{thm: cr} we get 
$X \setminus C \subseteq X \setminus_F C$, hence $X \setminus C = X \setminus_F C$. As $X \setminus C$ is an open
subgroup, by Proposition~\ref{prp: closed2} we have that $X \setminus C$ is also closed.
\end{proof}

Notice that classically a subgroup $C$ of a topological group is either clopen or has empty interior. 
By replacing the hypothesis of ``non-empty interior of $C$'' with the positive ``existence of an inhabited 
open subset of $C$'', constructively we have the following corollary.

\begin{corollary}\label{cor: corchar2}
Let $C$ be a subgroup of $X$. If $O$ is an inhabited open set in $X$ such that $O \subseteq C$, then $C$ is clopen.
\end{corollary}

\begin{proof}
By Theorem~\ref{thm: chars}(ii) $C$ is open, and by Proposition~\ref{prp: closed2} $C$ is also closed.
\end{proof}

% 
% 
% 
% 
% 
% \section{Concluding remarks}
% \label{sec: concluding}
% 

Here we have presented some very first, fundamental results in the theory of Bishop topological groups.
Clearly, we can work similarly for other algebraic structures, like rings and modules, equipped with a compatible
Bishop topology. There is a plethora of open questions related to Bishop topological groups. The complete regularity of 
a Bishop topological group, the study of ``compact'' subsets of Bishop topological groups, 
the uniform continuity of the elements of $F$ when $X$ is a compact Bishop topological group with $F$,
and the interpretation of local compactness in the theory of Bishop topological groups, are some of the numerous topics
to which our future work could hopefully be directed.\\[2mm]

% C + K
% 
% 
% 
% 
% compactness,  local compactness 
% 
% uniform continuity p.70 Arhangelski
% 

\noindent
\textbf{Acknowledgements}: The content of this paper was part of my talk in ``Conference Algebra and Algorithms''
that took place in Djerba, Tunisia, in February 2020. I would like to thank the organisers, Peter Schuster 
and Ihsen Yengui, for inviting me. This research 
%publication
was supported by LMUexcellent, funded by the Federal
Ministry of Education and Research (BMBF) and the Free State of Bavaria under the
Excellence Strategy of the Federal Government and the L\"ander.

%\bibliography{Petrakis}

\end{document}